\documentclass{article}
\usepackage{amsmath, amssymb, amsthm, graphicx, amscd, amsfonts, amscd}

\usepackage[pagewise, displaymath, mathlines]{lineno}

\usepackage{setspace}
\usepackage{microtype} % Slightly tweak font spacing for aesthetics

\newtheorem{theorem}{Theorem}[section]
\newtheorem{lemma}[theorem]{Lemma}
\newtheorem{proposition}[theorem]{Proposition}
\newtheorem{corollary}[theorem]{Corollary}

\def\bb #1{ {\mathbb #1} }
\def\c #1{ {\mathcal #1} }
\def\f #1{ {\mathfrak #1} }

\def\b #1{ {\bf #1} }

\newcommand\SkH { {\c S_k} }
\newcommand\oSkH { {\overline{\c S}_k} }

\newcommand{\uu}{ {\underline{u}} }

\newcommand{\ui}{ {\underline i} }
\newcommand{\uj}{ {\underline j} }

\newcommand{\bj}{ {\bf{j}} }

\newcommand{\oO}{ {\overline{\c O}} }
\newcommand{\oC}{ {\overline{\c C}} }
\newcommand{\oD}{ {\overline{D}} }
\newcommand{\oG}{ {\overline{G}} }
\newcommand{\tG}{ {\widetilde{G}} }

\DeclareMathOperator {\Teich}{Teich}
\DeclareMathOperator {\Sym}{Sym}
\DeclareMathOperator {\ord}{ord}
\DeclareMathOperator {\Fil}{Fil}
\DeclareMathOperator {\Frob}{Frob}
\DeclareMathOperator {\Tr}{Tr}
\DeclareMathOperator {\Kl}{Kl}

\begin{document}
\title{Symmetric Power $L$-functions of the hyper-Kloosterman Family}

\author{C. Douglas Haessig\footnote{This work was partially supported by a grant from the Simons Foundation \#314961.} \  and Steven Sperber}

\date{\today}
\maketitle
\begin{abstract}
The symmetric power $L$-function of the hyper-Kloosterman family is a rational function over the integers. Its degree and complex absolute values of its zeros and poles are now known through the work of Fu and Wan. The purpose of this paper is to study the $p$-adic absolute value of these zeros and poles. In particular, we give a uniform lower bound, independent of the symmetric power, of the $q$-adic Newton polygon of this $L$-function under suitable conditions. We also give similar results for any other linear algebra operation of the hyper-Kloosterman family, such as tensor, exterior, symmetric powers, or combinations thereof.
\end{abstract}

\tableofcontents

%%%-----------------------------------------------------------------------
\section{Introduction}

Hyper-Kloosterman sums are exponential sums associated to the $n$-variable Laurent polynomial $x_1 + \cdots + x_n + \frac{t}{x_1 \cdots x_n}$. Here, $t$ is a parameter, and as it varies over different values gives us a family of exponential sums. There are various ways to study this family as a whole. One such way is the symmetric power $L$-function, a study initiated by Robba \cite{Robba-SymmetricPowersof-1986} in the one-variable Kloosterman case ($n = 1$) using $p$-adic methods, and later by Fu and Wan \cite{MR2194679, FuWan-$L$-functionssymmetricproducts-2005, FuWan-TrivialfactorsL-functions-} for general $n$ using $\ell$-adic methods. The authors of the present paper studied generalized Kloosterman sums \cite{MR3572279}, which includes the hyper-Kloosterman family as a special case, and $L$-functions associated to linear algebraic operations such as tensor, exterior, and symmetric powers. Missing from this body of work is (1) a $p$-adic cohomological description of the symmetric power $L$-function of the hyper-Kloosterman family, and (2) $p$-adic estimates of its roots. This paper answers the first, and under a suitable  hypothesis on the connection also gives $p$-adic estimates. The latter generalizes the first author's $p$-adic estimates  \cite{MR3694643, MR4314031} for the one-variable Kloosterman sum. Let us describe this more precisely.

Throughout the paper we fix a non-trivial additive character $\Theta$ on the finite field $\bb F_q$, and let $q = p^a$ with $p$ any prime number. Let $\bar t \in \overline{\bb F}_q^*$ and denote its degree by $\deg(\bar t) := [\bb F_q(\bar t) : \bb F_q]$. Set $q_{\bar t} := q^{\deg(\bar t)}$ and write $\bb F_{q_{\bar t}} := \bb F_q(\bar t)$. We denote by $\bb F_{q_{\bar t}^m}$ the degree $m$ extension field of $\bb F_{q_{\bar t}}$. We extend the additive character to $\bb F_{q_{\bar t}}$ by setting $\Theta_{\bar t} := \Theta \circ \Tr_{\bb F_{q_{\bar t}} / \bb F_q}$. For each positive integer $m$, define the hyper-Kloosterman  exponential sum by
\[
\Kl_n(\bar t, m) := \sum_{\bar x \in (\bb F_{q_{\bar t}^m} )^n} \Theta_{\bar t} \circ \Tr_{\bb F_{q_{\bar t}^m} / \bb F_{q_{\bar t}}}( \bar x_1 + \cdots + \bar x_n + \frac{\bar t}{\bar x_1 \cdots \bar x_n}).
\]
It is well-known that its associated $L$-function
\[
L(\Kl_n / \bb F_q, \bar t, T) := \exp\left( \sum_{m=1}^\infty \Kl_n(\bar t, m) \frac{T^m}{m} \right)
\]
can be written in the form
\[
L(\Kl_n / \bb F_q, \bar t, T)^{(-1)^{n+1}} =  (1 - \pi_0(\bar t) T) \cdots (1 - \pi_n(\bar t) T) \in \bb Z[\zeta_p][T].
\]
A lot is known about these roots \cite{AS-hyperKloo-revisited, MR463174, MR558257}. For example, they are all $q_{\bar{t}}$-Weil numbers meaning their complex absolute values, as well as all of their Galois conjugates, equal $q_{\bar{t}}^{n/2}$. Also, if $\ell \neq p$ is a prime number then each $\pi_i(\bar{t})$ is an $\ell$-adic unit, and $p$-adically they may be ordered so that  $\ord_{q_{\bar{t}}} \pi_i(\bar t) = i$ for each $i = 0, 1, \ldots, n$. 

For $k$ a positive integer, define the $k$-th symmetric power $L$-function of the hyper-Kloosterman family by 
\[
L(\Sym^k \Kl_n / \bb F_q, T) := \prod_{\bar t \in | \bb G_m^n / \bb F_q |} \prod_{i_0 + \cdots + i_n = k} \frac{1}{1 - \pi_0(\bar t)^{i_0} \cdots \pi_n(\bar t)^{i_n} T^{\deg(\bar t)}},
\]
where the first product runs over all closed points of the algebraic torus defined over $\bb F_q$ and the second product has each $i_j \geq 0$. Using $\ell$-adic techniques, Fu-Wan \cite{MR2194679, FuWan-TrivialfactorsL-functions-} proved that this $L$-function is a rational function over $\bb Z$ which takes the form
\[
L(\Sym^k \Kl_n / \bb F_q, T) = P(n, k, T) \cdot M(n, k, T),
\]
where $P(n, k, T)$ is a rational function over $\bb Z$, the so-called trivial factor, whose reciprocal zeros and poles are all powers of $q$ (see Section \ref{S: trivial factor} below for its precise description), and $M(n, k, T)$ is a polynomial with $\bb Z$-coefficients whose reciprocal roots are $q^{kn+1}$-Weil numbers; that is, algebraic integers all of whose Galois conjugates have complex absolute value $q^{(kn+1)/2}$. In most cases, the $L$-function is a polynomial. In fact, only when $p$, $n$ and $k$ are all even will the $L$-function in reduced form have a denominator that is precisely $1 - q^{(kn+2)/2} T$.

The main results of this paper are to compute the associated $p$-adic cohomology spaces of $L(\Sym^k \Kl_n / \bb F_q, T)$, demonstrate that cohomology is acyclic under suitable conditions, and study the action of Frobenius on cohomology to obtain $p$-adic estimates of the roots of the $L$-function. To summarize these results, suppose $\gcd(n+1, p) = 1$, let $\zeta_{n+1}$ be a primitive $(n+1)$-root of unity in $\overline{\bb F}_p$, and define
\[
d_k(n, p) := \# \{ (i_0, \ldots, i_n) \in \bb Z_{\geq 0}^{n+1} \mid i_0 + \cdots + i_n = k \text{ and }  \sum_{j=0}^n i_j \zeta_{n+1}^j \equiv 0 \text{ in } \overline{\bb F}_p \}.
\]
Next, set $R(T) := \sum T^{i_1 + 2 i_2 + \cdots + n i_n}$ where the sum runs over all $(i_0, \ldots, i_n) \in \bb Z_{\geq 0}^{n+1}$ satisfying $i_0 + \cdots + i_n = k$. Write $R(T) / (1 + T + \cdots + T^n) = \sum_{i=0}^\infty h_i(n, k) T^i$. 

\begin{theorem}\label{T: main thm}
Suppose $d_k(n, p) = 0$. Then $L(\Sym^k \Kl_n / \bb F_q, T)$ is a polynomial in $1 + T \bb Z[T]$ of degree $\frac{1}{n+1} \binom{n+k}{n}$ whose $q$-adic Newton polygon lies on or above the $q$-adic Newton polygon of $\prod (1 - q^i T)^{h_i(n, k)}$.
\end{theorem}

In the case $n = 1$, the classical 1-variable Kloosterman family, Fres\'an-Sabbah-Yu \cite{Fresan2018HodgeTO} proved that the $h_i$ above equal Hodge numbers coming from a related exponential mixed Hodge structure. We expect that the $h_i$ for general $n$ will have a similar interpretation, and thus Theorem \ref{T: main thm} may be thought of as a continuation of the mantra ``Newton over Hodge''. 

Symmetric power $L$-functions of the hyper-Kloosterman family were first studied in a series of papers by Fu and Wan \cite{MR2194679, FuWan-$L$-functionssymmetricproducts-2005, MR2425148, FuWan-TrivialfactorsL-functions-}, where they calculate its degree, compute its trivial factors, and give a motivic interpretation. The current authors studied in \cite{MR3572279} such $L$-functions but for more general families which include the hyper-Kloosterman as a special case.

For symmetric powers of the classical 1-variable Kloosterman sums, more is known. These were first studied by Robba \cite{Robba-SymmetricPowersof-1986} who proved among other things that these are always polynomials, and gave a conjectural degree. Fu-Wan \cite{MR2194679} proved a corrected formula for the degree and asked whether there exists a uniform quadratic lower bound for the Newton polygon of the $L$-function. This was shown to be true independently by \cite{Fresan2018HodgeTO} and \cite{MR3694643, MR4314031}. For $k$ small, Yun gave an automorphic interpretation of $L(\Sym^k \Kl_1 / \bb F_q, T)$. In the other direction, when $p$ is small, the following modular interpretation \cite{MR4314031} holds for all $k$: for $p = 2$ let $\Gamma$ to be the congruence subgroup $\Gamma := \Gamma_1(4)$, and for $p = 3$ set $\Gamma := \Gamma_1(3)$. Then for every $k \geq 1$,
\[
L(\Sym^k \Kl_1, s) = (1 - s) \det(1 - s U_p \mid S_{k+2}(\Gamma)),
\]
where $U_p$ is the Atkin $U$-operator acting on cusp forms $S_{k+2}(\Gamma)$ of weight $k+2$. It was also shown that, for $p = 2$ and $k$ odd, the Newton polygon equals the lower bound in Theorem \ref{T: main thm}. However, equality is not expected in general. 

The methods in this paper are not limited to symmetric products, and may be easily adjusted to suit other linear algebra operations such as tensor, symmetric, and exterior products, or combinations thereof. Let us describe how to obtain such a result. Let $V$ be a vector space of dimension $n+1$ with basis $\{ v_0, \ldots, v_n\}$. Let $\c L$ be a linear algebra operation. This means $\c L V$ is a quotient space of $V^{\otimes m}$ for some $m$. When needed for clarity, we will write $m= m_{\c L}$ and refer to $m_{\c L}$ as the total degree of $\c L V$. For instance, if $\c L = \Sym^k \otimes \wedge^\ell$, then $\c L V = \Sym^k V \otimes \wedge^\ell V$ is a quotient space of $V^{\otimes m_{\c L}}$ with $m_{\c L} = k + \ell$. For $\c L$ a linear algebra operation, $\c LV$ will be a vector space with basis $\c B_{\c L} = \{ v_{j_1} \otimes \cdots \otimes v_{j_m} \in \c L V \mid (j_1, \ldots, j_m) \in J \}$ for some index set $J \subseteq \bb Z_{\geq 0}^m$. We define the $\c L$-product $L$-function of the hyper-Kloosterman family over $\bb F_q$ by
\[
L(\c L \Kl_n / \bb F_q, T) := \prod_{\bar t \in | \bb G_m^n / \bb F_q |} \prod_{(j_1, \ldots, j_m) \in J} \frac{1}{1 - \pi_{j_1}(\bar t) \cdots \pi_{j_m}(\bar t) T^{\deg(\bar t)}}.
\]
Define
\[
d(\c L, n, p) := \{ (j_1, \ldots, j_m) \in J \mid \sum_{i = 1}^m \zeta_{n+1}^{j_i} = 0 \text{ in } \overline{\bb F}_p \}.
\]
Set $R_{\c L}(T) := \sum T^{j_1 + j_2 + \cdots + j_m}$ where the sum runs over all $(j_1, \ldots, j_m) \in J$. Write $R_{\c L}(T) / (1 + T + \cdots + T^n) = \sum_{i=0}^\infty h_i(\c L, n) T^i$.

\begin{theorem}
Suppose $d(\c L, n, p) = 0$. Then $L(\c L \Kl_n / \bb F_q, T)$ is a polynomial of degree $\frac{1}{n+1} \# J$ whose $q$-adic Newton polygon lies on or above the $q$-adic Newton polygon of $\prod (1 - q^i T)^{h_i(\c L, n)}$.
\end{theorem}

For example, this result holds for all linear algebra operations $\c L$ on the classical 1-variable Kloosterman family ($n = 1$) when $m_{\c L}$ is odd and $p > m_{\c L}$.

Last, we take note of some references that studied other families aside from the hyper-Kloosterman. For general multi-parameter families see \cite{MR3239170, MR3572279}; for moment $L$-functions of Calabi-Yau hypersurfaces see \cite{MR2385246}; for the Airy family \cite{MR2542216, MR2873140, SabbahYu-Airy}; and for the Legendre family of elliptic curves see \cite{Adolphson-$p$-adictheoryof-1976, Dwork-Heckepolynomials-1971, Haessig-EllCurve}.

{\bf Acknowledgements.} We thank Andy O'Desky, Vic Reiner, and Dennis Stanton for some very helpful discussions.

%%%-----------------------------------------------------------------------
\section{The Trivial Factor $P(n, k, T)$}\label{S: trivial factor}

In this section we recall the results of Fu and Wan in \cite{MR2194679, FuWan-TrivialfactorsL-functions-} determining the trivial factor $P(n, k, T)$.  This factor may be broken down into three parts:
\[
P(n, k, T) = \frac{A_0(n, k, T) A_\infty(n, k, T)}{B(n, k, T)},
\]
with each factor on the right being described in a moment. Before giving their description, we point out that $P(n, k, T)$ is a polynomial except when $p$, $k$, and $n$ are all even, and in this case, the denominator of $P(n, k, T)$, after reduction, is a single linear term of the form $(1 - q^{(kn + 2)/2} T)$.

Define the non-negative integers $m_k(i)$ by the expansion
\[
\frac{(1 - x^{n+1}) \cdots (1 - x^{n+k-1})(1 - x^{n+k})}{(1 - x^2) \cdots (1 - x^{k-1})(1 - x^k)} = \sum_{i = 0}^\infty m_k(i) x^i.
\] 
Using this, define the factor $A_0(n, k, T)$ ``at the origin'' by

\[
A_0(n, k, T) := \prod_{i=0}^{\lfloor \frac{kn}{2} \rfloor} (1 - q^i T)^{m_k(i)}.
\] 
This factor was computed in \cite[Theorem 0.1]{FuWan-TrivialfactorsL-functions-}. 

 The  factor $A_\infty(n, k, T)$  was computed in \cite[Theorem 2.5]{MR2194679} and takes various forms depending on the parity of $n$ and $k$. We will use notation from \cite[Section 2]{MR2194679} however, note that we use a different definition of $n$ in our paper; our $n$ corresponds to $n+1$ in their paper. Let $\zeta$ be a primitive $(n+1)$-th root of unity in $\overline{\bb F}_q$. Set
\[
S_k(n, p) := \{ \ui := (i_0, i_1, \ldots, i_n) \in \bb Z_{\geq 0}^{n+1} \mid i_0 + \cdots + i_n = k, \text{ and } i_0 + i_1 \zeta + \cdots + i_n \zeta^n \equiv 0 \text{ in } \overline{\bb F}_q \}.
\]
Observe that the shift operator $\sigma$ defined by
\[
\sigma(i_0, i_1, \ldots, i_n) := (i_n, i_0, i_1, \ldots, i_{n-1})
\]
is stable on $S_k(n, p)$. Denote by $a_k(n, p)$ the number of $\sigma$-orbits in $S_k(n, p)$. Then
\[
A_\infty(n, k, T) = 
\begin{cases}
1 & \text{if $n$ is odd and $k$ odd}  \\
(1 - q^{kn/2} T)^{a_k(n, p)} & \text{if $n$ is even, and any $k$}.
\end{cases}
\]
Suppose now that $n$ is odd and $k$ is even. Let $V$ be a vector space over a field of characteristic zero with basis $\{e_0, \ldots, e_n\}$. Denote by $\Sym^k V$ the $k$-th symmetric power vector space of $V$. Using the multi-index notation $e^\ui = e_0^{i_0} \cdots e_n^{i_n}$ with $\ui = (i_0, \ldots, i_n) \in \bb Z_{\geq 0}^{n+1}$, a basis for this space is $\{ e^\ui \mid i_0 + \cdots + i_n = k \}$. Let $\f o$ be an orbit in $S_k(n, p)$, and $\ui \in \f o$. Define
\[
v_{\f o} := v_\ui := \sum_{\ell = 0}^n (-1)^{i_n + \cdots + i_{n - \ell}} e^{\sigma^\ell(\ui)}.
\]
Since $k$ is even, $v_{\sigma \ui} = (-1)^{i_n} v_\ui$ and thus $v_{\f o}$ is well-defined (it only depends on the orbit, not on the choice of $\ui$). Let $b_k(n, p)$ denote the number of orbits $\f o$ such that $v_{\f o} \not= 0$. Define the weight $w(\ui) := i_1 + 2 i_2 + \cdots + n i_n$. Observe that since $n+1$ and $k$ are even, then if $w(\ui)$ is odd then $w(\sigma \ui)$ is odd, and thus all elements in an orbit have odd weight, or all even weights. We say an orbit $\f o$ is odd if $w(\ui)$ is odd for some, and hence all, $\ui \in \f o$. Let $c_k(n, p)$ denote the number of odd orbits $\f o$ such that $v_{\f o} \not= 0$. Then, in the case of both $n$ is odd and $k$ is even,
\begin{align*}
A_\infty&(n, k, T) \\
&=
\begin{cases}
(1 - q^{kn/2} T)^{b_k(n, p)} & \text{if } 2(n+1) \mid (q-1) \\
(1 + q^{kn/2} T)^{c_k(n, p)}(1 - q^{kn/2} T)^{b_k(n, p) - c_k(n, p)} & \text{if } 2(n+1) \nmid (q-1) \text{ and either } 4 \mid (n+1) \text{ or } 4 \mid k \\
(1 - q^{kn/2} T)^{c_k(n, p)}(1 + q^{kn/2} T)^{b_k(n, p) - c_k(n, p)} & \text{if } 2(n+1) \nmid (q-1), 4 \mid (n+1) \text{ and } 4 \mid k.
\end{cases}
\end{align*}

The factor $B(n, k, T)$ was computed in \cite[Theorem 0.3]{FuWan-TrivialfactorsL-functions-} and equals 1 in most cases: 
\[
B(n, k, T) =
\begin{cases}
(1 - q^{kn/2} T)(1 - q^{(kn+2)/2} T) & \text{if $p$, $k$, and $n$ are all even} \\
1 & \text{otherwise}.
\end{cases}
\]
Notice that the factor $(1 - q^{kn/2} T)$ will cancel with the identical term which appears in $A_\infty(n, k, T)$ since $a_k(n, p) \geq 1$. On the other hand, the factor $(1 - q^{(kn+2)/2} T)$ in $B(n, k, T)$ cannot be cancelled in $L(\Sym^k \Kl_n / \bb F_q, T)$ since all the reciprocal roots in $M(n, k, T)$ have complex absolute value equal to $q^{(kn+1)/2}$, and no reciprocal root of $A_0$ or $A_\infty$ takes this value. Thus, the $L$-function in reduced form is indeed a rational function with one linear term in the denominator.

%%%-----------------------------------------------------------------------
\section{Cohomological formula}\label{S: cohom formula}

In this section, we construct cohomology spaces $H^1$ and $H^0$ with Frobenius maps $\bar \beta_k$, and prove the following cohomological formula.

\begin{theorem}\label{T: cohom formula}
\[
L(\Sym^k \Kl_n / \bb F_q, T) = \frac{\det(1 - \bar \beta_k T \mid H^1)}{\det(1 - q \bar \beta_k T \mid H^0)}.
\]
\end{theorem}

The rest of this section is devoted to a proof of this result. Fix a prime $p$ and let $q = p^a$ where $a \geq 1$. Let $\bb Q_q$ be the unramified extension of $\bb Q_p$ of degree $a$ and denote by $\bb Z_q$ its ring of integers. Let $E(T) := \text{exp}(\sum_{j=0}^{\infty} \frac{T^{p^j}}{p^j}) \in (\mathbb{Q} \cap \mathbb{Z}_p)[[T]]$ be the Artin-Hasse series, and let $\gamma \in \overline{\mathbb{Q}}_p$ be a root of $\sum_{j=0}^{\infty} \frac{T^{p^j}}{p^j}$ such that $\ord_p \gamma = 1/(p-1)$. Fix a $q$-th root of $\gamma$, and denote by $\f O_q$ the ring of integers of $\bb Q_q(\gamma^{1/q})$; denote its residue field by $\bb F_q$.

Denote by $\Delta \subset \bb R^n$ the polytope defined by the convex hull of the standard basis $(1, 0, \ldots, 0)$ through $(0, \ldots, 0, 1)$, along with the point $(-1, -1, \ldots, -1)$; this is the Newton polytope of $x_1 + \cdots + x_n + t/(x_1 \cdots x_n)$ viewed as a Laurent polynomial in the variables $\{x_1, \cdots, x_n \}$. For $u = (u_1, \ldots, u_n) \in \bb Z^n$ define 
\begin{align*}
m(u) &:=  \max\{0, -u_1, \ldots, -u_n\} \\
w(u) &:= u_1 + \cdots + u_n + (n+1) m(u).
\end{align*}
We note that $w$ is the usual polyhedral weight function coming from $\Delta$. The following properties hold:

\begin{proposition}\label{P: w and m}
For every $u, v \in \bb Z^n$ and $r \in \bb Z_{\geq 0}$, we have
\begin{enumerate}
\item $w(u) = 0$ if and only if $u = 0$.
\item $w(ru) = r w(u)$. 
\item $w(u + v) \leq w(u) + w(v)$, with equality if and only if $u$ and $v$ are cofacial with respect to $\Delta$.
\item Properties 2 and 3 also hold for $m(u)$.
\end{enumerate}
\end{proposition}

Dwork's splitting  function $\theta(T) :=  E(\gamma T) =\sum_{j=0}^{\infty} \theta_j T^j$ has coefficients in $\bb Z_p[\gamma]$ satisfying $\text{ord}_p(\theta_j) \geq \frac{j}{p-1}$. Define
\[
F(t, x) := \theta(x_1) \cdots \theta(x_n) \cdot \theta(t/(x_1 \cdots x_n)), 
\]
and note that we may write $F(t, x) = \sum_{u \in \bb Z^n} B(u) x^u$ with
\begin{equation}\label{E: Frob coeff est}
B(u) := \sum \theta_{m_1} \theta_{m_2} \cdots \theta_{m_n} \theta_l t^l,
\end{equation}
where the sum runs over all $(m_1, \ldots, m_n, l) \in \bb Z_{\geq 0}^{n+1}$ such that $l \geq m(u)$ and $m_i - l = u_i$ for $1 \leq i \leq n$.

Define the spaces
\[
\c O_{p^k} := \left\{ \xi := \sum_{r \geq 0} a(r) \gamma^{r (n+1) / p^k} t^r \mid a(r) \in \f O_q, a(r) \rightarrow 0 \text{ as } \ r \rightarrow \infty \right \},
\]
and
\[
\c C_{p^k} := \left\{ \zeta := \sum_{u \in \bb Z^n} \xi(u) \gamma^{w(u)} t^{p^k m(u)} x^u \mid \xi(u) \in \c O_{p^k}, \ \xi(u) \rightarrow 0 \text{ as } w(u) \rightarrow \infty \right \}.
\]
We will often write $\c O$ and $\c C$ when $k = 0$. With the sup-norms $\| \xi \| := \sup |a(r)|$ and $\| \zeta \| := \sup \| \xi(u) \|$ these are $p$-adic Banach modules over $\f O_q$ and $\c O_{p^k}$, respectively. The convexity of the functions $w(u)$ and $m(u)$ mentioned above in fact shows that the modules are in fact Banach algebras over $\f O_q$ and $\c O_{p^k}$ respectively.

\bigskip\noindent{\bf Relative cohomology.}  Define 
\[
G(t, x) := \prod_{j=0}^\infty F(t^{p^j}, x^{p^j}),
\]
and note that $F(t, x) = \frac{G(t, x)}{G(t^p, x^p)}$. For $i \geq 0$ set $\gamma_i := \sum_{j = 0}^i \gamma^{p^j} / p^j = -\sum_{j = i+1}^\infty \gamma^{p^j}/p^j$. The latter shows $\ord_p(\gamma_i) = \frac{p^{i+1}}{p-1} - (i+1)$. Set $K(t, x) := x_1 + \cdots + x_n + t/(x_1 \cdots x_n)$, and define $H(t, x) := \sum_{i = 0}^\infty \gamma_i K(t^{p^i}, x^{p^i})$. For each $1 \leq l \leq n$, define the differential 
\begin{align*}
D_{t^{p^k}, l} :&= \frac{1}{G(t^{p^k}, x)} \circ x_l \frac{\partial}{\partial x_l} \circ G(t^{p^k}, x) \\
&= x_l \frac{\partial}{\partial x_l} + x_l \frac{\partial H(t^{p^k}, x)}{\partial x_l}.
\end{align*}
One checks that $\gamma_i p^i x_l \frac{\partial K(t^{p^{i+k}}, x^{p^i})}{\partial x_l} \in \c C_{p^k}$ for all $i \geq 0$. It follows that $x_l \frac{\partial H(t^{p^k}, x)}{\partial x_l} \in \c C_{p^k}$ so that $D_{t^{p^k}, l}$ is an endomorphism of $\c C_{p^k}$ over $\c O_{p^k}$. As the $D_{t^{p^k}, l}$ commute with one another we may construct a Koszul complex:  set $\Omega^0(\c C_{p^k}, D_{t^{p^k}}) := \c C_{p^k}$, and for $1 \leq i \leq n$,
\[
\Omega^i(\c C_{p^k}, D_{t^{p^k}}) := \bigoplus_{1 \leq j_1 < j_2 < \cdots < j_i \leq n} \c C_{p^k} \frac{dx_{j_1}}{x_{j_1}} \wedge \cdots \wedge \frac{dx_{j_i}}{x_{j_i}}
\]
with boundary map $D_{t^{p^k}}: \Omega^i \rightarrow \Omega^{i+1}$ defined by
\[
D_{t^{p^k}}(\xi \frac{d x_{j_1}}{x_{j_1}} \wedge \cdots \wedge \frac{d x_{j_i}}{x_{j_i}}) := \left( \sum_{l=1}^n D_{t^{p^k}, l}(\xi) \frac{d x_l}{x_l} \right) \wedge \frac{d x_{j_1}}{x_{j_1}} \wedge \cdots \wedge \frac{d x_{j_i}}{x_{j_i}}.
\]
Denote by $H^i(\c C_{p^k}, D_{t^{p^k}})$ the associated cohomology spaces of the complex $\Omega^\bullet(\c C_{p^k}, D_{t^{p^k}})$. It was shown in \cite{MR3572279} that $H^i = 0$ for every $i \not= n$, and $H^n(\c C_{p^k}, D_{t^{p^k}})$ is a free $\c O_{p^k}$-module of rank $n+1$. Furthermore, if $\xi_0(t, x), \ldots, \xi_n(t, x) \in \c C$ is a basis of $H^n(\c C, D_t)$, then $\xi_0(t^{p^k}, x), \ldots, \xi_n(t^{p^k}, x) \in \c C_{p^k}$ is a basis of $H^n(\c C_{p^k}, D_{t^{p^k}})$.

\bigskip\noindent{\bf Relative Frobenius.} The function $\psi_x: \sum a(u) x^u \mapsto \sum a(pu) x^u$ defines a map $\psi_x: \c C_{p^k} \rightarrow \c C_{p^{k+1}}$. Define
\begin{align*}
\alpha_1(t^{p^k}) :&= \frac{1}{G(t^{p^{k+1}}, x)} \circ \psi_x \circ G(t^{p^k}, x) \\
&= \psi_x \circ F(t^{p^k}, x).
\end{align*}
It follows from (\ref{E: Frob coeff est}) (see also Lemma \ref{L: chain level est} below) that $\alpha_1(t^{p^k}): \c C_{p^{k}} \rightarrow \c C_{p^{k+1}}$. Next, define $\alpha_a: \c C \rightarrow \c C_q$ by
\begin{align*}
\alpha_a :&= \alpha_1(t^{p^{a-1}}) \cdots \alpha_1(t^p) \alpha_1(t) \\
&= \frac{1}{G(t^q, x)} \circ \psi_x^a \circ G(t, x).
\end{align*}
Since $q D_{l, t^q} \circ \alpha_a = \alpha_a \circ D_{l, t}$ for each $1 \leq l \leq n$, we may define a chain map $\Frob^\bullet(\alpha_a):  \Omega^\bullet(\c C, D_t) \rightarrow \Omega^\bullet(\c C_q, D_{t^q})$ by
\begin{equation}\label{E: Fr chain map}
\Frob^i(\alpha_a) := \sum_{1 \leq j_1 < \cdots < j_i \leq n} q^{n-i} \alpha_a \frac{dx_{j_1}}{x_{j_1}} \wedge \cdots \wedge \frac{dx_{j_i}}{x_{j_i}}.
\end{equation}
Since $H^i = 0$ for every $i \not= n$, we need only focus on $\bar \alpha_a(t) := \Frob^n(\alpha_a): H^n(\c C, D_t) \rightarrow H^n(\c C_q, D_{t^q})$.

\bigskip\noindent{\bf Kloosterman $L$-function on the fibres.} Let $\bar t \in \overline{\bb F}_q^*$ with $\deg(\bar t) := [\bb F_q(\bar t) : \bb F_q]$. Using Dwork's splitting function $\theta$ define the additive character $\Theta: \bb F_q \rightarrow \overline{\bb Q}_p$ via $\Theta := \theta(1)^{\Tr_{\bb F_q / \bb F_p}(\cdot)}$, and $\Theta_{\bar t} := \Theta \circ \Tr_{\bb F_q(\bar t) / \bb F_q}(\cdot)$. For each positive integer $m$ define the Kloosterman sum over the base field $\bb F_q(\bar t)$:
\[
\Kl_n(\bar t, m) := \sum_{\bar x \in (\bb F_{q^{m \deg(\bar t)}}^* )^n} \Theta_{\bar t} \circ \Tr_{\bb F_{q^{m \deg(\bar t)}} / \bb F_q(\bar t)}( \bar x_1 + \cdots + \bar x_n + \frac{\bar t}{\bar x_1 \cdots \bar x_n}),
\]
and its associated $L$-function
\[
L(\Kl_n / \bb F_q, \bar t, T) := \exp\left( \sum_{m=1}^\infty \Kl_n(\bar t, m) \frac{T^m}{m} \right).
\]

Let $\hat t$ be the Teichm\"uller representative of $\bar t$, and denote by $\c O_{\hat t}$ the ring $\c O_{q^{\deg(\bar t)}}|_{t = \hat t}$, that is, the ring $\c O_{q^{\deg(\bar t)}}$ specialized at $t = \hat t$. Specialization of $t \mapsto \hat t$ is a ring map from $\c O_{q^{\deg(\bar t)}}$ to $\c O_{\hat t}$. Similarly, denote by $\c C_{\hat t}$ the Banach $\c O_{\hat t}$-module obtained by specializing the space $\c C_{q^{\deg(\bar t)}}$ at $t = \hat t$. Last, set $\alpha_{\hat t} := \alpha_{a \deg(\bar t)}(\hat t)$ and note it is an endomorphism of $\c C_{q^{\deg(\bar t)}}$.

Next, denote by $\Omega^\bullet(\c C_{\hat t}, D_{\hat t})$ the complex obtained by specializing the complex $\Omega^\bullet(\c C_{q^{\deg(\bar t)}}, D_{t^{q^{\deg(\bar t)}}})$ at $t = \hat t$, and define the chain map $\Frob^\bullet( \alpha_{\hat t})$ on $\Omega^\bullet(\c C_{\hat t}, D_{\hat t})$ as in (\ref{E: Fr chain map}). Dwork's trace formula (the first equality) along with the chain map (the second equality) tells us that
\begin{align*}
\Kl_n(\bar t, m) &= (q^{m \deg(\bar t)} - 1)^n \Tr( \alpha_{\hat t} \mid \c C_{\hat t}) \\
&= \sum_{i = 0}^n (-1)^i \Tr( \Frob^i( \alpha_{\hat t} )^m \mid H^i( \c C_{\hat t}, D_{\hat t})).
\end{align*}
Since cohomology is acyclic except in the top dimension, setting $\bar \alpha_{\hat t} := \Frob^n( \alpha_{\hat t} )$, we see that
\[
\Kl_n(\bar t, m) = (-1)^n \Tr( \bar \alpha_{\hat t}^m \mid H^n( \c C_{\hat t}, D_{\hat t})).
\]
Further, since $H^n(\c C_{\hat t}, D_{\hat t})$ is of rank $n+1$ at each fibre $\hat t$,
\begin{align}\label{E: eigenvalues of alpha}
L(\Kl_n / \bb F_q, \bar t, T)^{(-1)^{n+1}} &= \det(1 - \bar \alpha_{\hat t} T \mid H^n( \c C_{\hat t}, D_{\hat t})) \notag \\
&= (1 - \pi_0(\bar t) T) \cdots (1 - \pi_n(\bar t) T)
\end{align}
where the reciprocal roots $\{ \pi_i(\bar t) \}_{i=0}^n$ are the eigenvalues of $\bar \alpha_{\hat t}$ acting on $H^n(\c C_{\bar t}, \c D_{\bar t})$

\bigskip\noindent{\bf Symmetric power cohomology.} Define the differential
\begin{align*}
\partial_{t^{p^m}} :&= \frac{1}{G(t^{p^m}, x)} \circ t \frac{\partial}{\partial t} \circ G(t^{p^m}, x) \\
&= t \frac{\partial}{\partial t} + t \frac{\partial H(t^{p^m}, x)}{\partial t} ,
\end{align*}
and note $\partial_{t^{p^m}}$ is an endomorphism of $\c C_{p^m}$. Since $\partial_{t^{p^m}}$ commutes with each $D_{t^{p^m}, l}$ it induces a connection map $\partial_{t^{p^m}}$ on $H^n(\c C_{p^m}, D_{t^{p^m}})$. Also, since $t \frac{d}{dt}$ and $\psi_x$ commute, we see that $\partial_{t^q} \alpha_a(t) = \alpha_a(t) \partial_t$, which carries over to the cohomology space $H^n$:
\begin{equation}\label{E: partial commute no p}
\partial_{t^q} \bar \alpha_a(t) = \bar \alpha_a(t) \partial_t.
\end{equation}

Consider now the $k$-th symmetric power of $H^n(\c C_q, D_{t^q})$ over $\c O_q$:
\[
\c S_{k, t^q} := \Sym^k_{\c O_q} H^n(\c C_q, D_{t^q}).
\]
Extend the map $\partial_{t^q}$ to this space as follows: elements in $\c S_{k, t^q}$ are linear combinations over $\c O_q$ of monomials of the form $u_1 \cdots u_k$, where $u_1, \ldots, u_k \in H^n(\c C_q, D_{t^q})$. Define $\partial_{t^q}$ by linearly extending over $\c O_q$ the action on monomials defined by
\[
\partial_{t^q}(u_1 \cdots u_k) := \sum_{i=1}^k u_1 \cdots \hat u_i \cdots u_k \partial_{t^q}(u_i),
\]
where $\hat u_i$ indicates that $u_i$ has been removed from the product. Define the complex $\Omega^\bullet(\c S_{k, t^q}, \partial_{t^q})$ by:
\[
\Omega^0 := \c S_{k, t^q} \qquad \text{and} \qquad \Omega^1 := \c S_{k, t^q} \frac{dt}{t},
\]
with boundary map $\partial_{t^q}: \Omega^0 \rightarrow \Omega^1$ defined by $\xi \mapsto \partial_{t^q}(\xi) \frac{dt}{t}$. Denote the associated cohomology spaces by $H^i(\c S_{k, t^q}, \partial_{t^q})$ for $i = 0, 1$.

\bigskip\noindent{\bf Cohomological formula.} For $1 \leq m \leq a$ define $\psi_t: \c O_{p^m} \rightarrow \c O_{p^{m-1}}$ by $\psi_t: \sum a(r) t^r \mapsto \sum a(pr) t^r$. Recall, if the image of $\xi_0(t, x), \ldots, \xi_n(t, x) \in \c C$ is a basis of $H^n(\c C, D_t)$ then the image of $\xi_0(t^q, x), \ldots, \xi_n(t^q, x) \in \c C_q$ is a basis of $H^n(\c C_{q}, D_{t^q})$. Let $\bar \xi_0(t, x), \ldots, \bar \xi_n(t, x)$, with $\xi_i(t, x) \in \c C$, be any basis of $H^n(\c C, D_t)$, then $\{ \bar \xi_0^{i_0} \cdots \bar \xi_n^{i_n} \mid i_0 + \cdots + i_n = k \}$ is a basis of $\c S_{k, t}$. Define $\psi_t: \c S_{k, t^{p^m}} \rightarrow \c S_{k, t^{p^{m-1}}}$ by linearly extending the action: for $\zeta \in \c O_{p^m}$ and $\zeta \bar \xi_0(t^{p^m}, x)^{i_0} \cdots \bar \xi_n(t^{p^m}, x)^{i_n} \in \c S_{k, t^{p^m}}$ define
\[
\psi_t( \zeta \bar \xi_0(t^{p^m}, x)^{i_0} \cdots \bar \xi_n(t^{p^m}, x)^{i_n} ) := \psi_t(\zeta) \bar \xi_0(t^{p^{m-1}}, x)^{i_0} \cdots \bar \xi_n(t^{p^{m-1}}, x)^{i_n}.
\]
From the definition,
\begin{align*}
\partial_{t^{p^m}}( &\zeta \bar \xi_0(t^{p^m}, x)^{i_0} \cdots \bar \xi_n(t^{p^m}, x)^{i_n} )  \\
&= t \frac{d}{dt}(\zeta)\bar \xi_0(t^{p^m}, x)^{i_0} \cdots \bar \xi_n(t^{p^m}, x)^{i_n} + \sum_{j=1}^n i_j \zeta \bar \xi_0(t^{p^m}, x)^{i_0} \cdots \bar \xi_j(t^{p^m}, x)^{i_j -1} \cdots \bar \xi_n(t^{p^m}, x)^{i_n} \partial_{t^{p^m}}( \bar \xi_j(t^{p^m}, x)).
\end{align*}
It follows that $\psi_t \circ \partial_{t^{p^m}} = p \partial_{t^{p^{m-1}}} \circ \psi_t$. The map $\bar \alpha_a(t): H^n(\c C, D_t) \rightarrow H^n(\c C_q, D_{t^q})$ defines a map $\Sym^k(\bar \alpha_a): \c S_{k, t} \rightarrow \c S_{k, t^q}$ as follows: linearly extend over $\c O$ the action on monomials $u_1 \cdots u_k \in \c S_{k, t}$, with $u_1, \ldots, u_k \in H^n(\c C, D_{t})$, defined by $\Sym^k(\bar \alpha_a)(u_1 \cdots u_k) := \bar \alpha_a(u_1) \cdots \bar \alpha_a(u_k)$. From (\ref{E: partial commute no p}) it follows that $\partial_{t^q} \circ \Sym^k(\bar \alpha_a) = \Sym^k(\bar \alpha_a) \circ \partial_t$. Similarly, the map $\bar \alpha_{\hat t}$ on $H^n( \c C_{\hat t}, D_{\hat t})$ defines an endomorphism $\Sym^k \bar \alpha_{\hat t}$ of $\c S_{k, \hat t} := \Sym^k_{\c O_{\hat t}} H^n(\c C_{\hat t}, D_{\hat t})$.

Set $\beta_k := \psi_t^a \circ \Sym^k(\bar \alpha_a): \c S_{k, t} \rightarrow \c S_{k, t}$. Observe that
\[
\beta_k \circ \partial_t = \psi_t^a \circ \Sym^k(\bar \alpha_a) \circ \partial_t = \psi_t^a \circ \partial_{t^q} \circ \Sym^k(\bar \alpha_a) = q \partial_t \circ \psi_t^a \circ \Sym^k(\bar \alpha_a) = q \partial_t \circ \beta_k.
\]
Hence, $\beta_k$ induces an endomorphism $\bar \beta_k$ on $H^i(\c S_{k, t}, \partial_t)$, and a chain map on the exact sequence:
\[
\begin{CD}
0 @>>> H^0 @>>> \c S_{k, t} @>\partial_t>> \c S_{k, t} @>>> H^1 @>>> 0 \\
@. @V q^m \bar\beta_k^mVV @Vq^m \beta_k^mVV @VV\beta_k^mV @VV\bar \beta_k^mV @. \\
0 @>>> H^0 @>>> \c S_{k, t} @>\partial_t>> \c S_{k, t} @>>> H^1 @>>> 0.
\end{CD}
\]
Using that the alternating sum of traces is zero, we have
\begin{equation}\label{E: Alt trace}
(q^m - 1) Tr(\beta_k^m \mid \c S_{k, t}) = q^m Tr(\bar \beta_k^m \mid H^0) - Tr(\bar \beta_k^m \mid H^1).
\end{equation}
Setting
\[
\bar \alpha_{a, m}(t) := \bar \alpha_a(t^{q^{m-1}}) \cdots \bar \alpha_a( t^q) \bar \alpha_a(t), 
\]
then Dwork's trace formula gives
\begin{equation}\label{E: Dw trace}
(q^m - 1) Tr(\beta_k^m \mid \c S_{k, t}) = \sum_{\hat t \in \overline{\bb Q}_p, \hat t^{q^{m}-1} = 1} \Tr\left(\Sym^k  \bar \alpha_{a, m}(\hat t)  \mid \c S_{k, \hat t} \right).
\end{equation}
Note, when $\hat t = \Teich(\bar t)$ with $r := \deg(\bar t)$, then $\bar \alpha_{a, sr}(\hat t) = \bar \alpha_{a, r}(\hat t)^s = \bar \alpha_{\hat t}^s$. Consequently, $\Sym^k \bar \alpha_{a, sr}(\hat t) = \left(\Sym^k \bar \alpha_{a, r}(\hat t) \right)^s =  \left(\Sym^k \bar \alpha_{\hat t} \right)^s$. We may now prove the main result of this section.

\begin{proof}[Proof of Theorem \ref{T: cohom formula}]
From (\ref{E: Alt trace}), (\ref{E: Dw trace}), and the observation after (\ref{E: Dw trace}), we see that:
\begin{align}\label{E: cohom formula}
\frac{\det(1 - \bar \beta_k T \mid H^1)}{\det(1 - q \bar \beta_k T \mid H^0)} &= \exp \sum_{m = 1}^\infty (q^m - 1) \Tr(\beta_k^m \mid \c S_k) \frac{T^m}{m} \notag \\
&=  \exp \sum_{m = 1}^\infty \sum_{\hat t^{q^{m-1}} = 1} \Tr\left( \Sym^k  \bar \alpha_{a, m}(\hat t)  \mid \c S_{k, \hat t} \right) \frac{T^m}{m} \notag \\
&= \exp \sum_{r = 1}^\infty \sum_{\substack{\bar t \in \overline{\bb F}_q^*, r = \deg(\bar t) \\ \hat t = \Teich(\bar t)}} \sum_{s = 1}^\infty \left( \Sym^k \bar \alpha_{a, r}(\hat t) \right)^s \frac{T^{s r}}{s r} \notag \\
&= \prod_{r \geq 1} \prod_{\substack{\bar t \in \overline{\bb F}_q^*, r = \deg(\bar t) \\ \hat t = \Teich(\bar t)}} \det(1 -  \Sym^k \bar \alpha_{\hat t} T^r \mid \c S_{k, \hat t})^{-1/r}.
\end{align}

From (\ref{E: eigenvalues of alpha}), $\bar \alpha_{\hat t}$ has eigenvalues $\pi_0(\bar t), \ldots, \pi_n(\bar t)$, and thus $\Sym^k \bar \alpha_{\hat t}$ has eigenvalues $\pi_0(\bar t)^{i_0} \cdots \pi_n(\bar t)^{i_n}$ for $i_0 + \cdots + i_n = k$. Hence,
\[
\det(1 -  \Sym^k \bar \alpha_{\hat t} T \mid \c S_{k, \hat t}) = \prod_{i_0 + \cdots + i_n = k} (1 - \pi_0(\bar t)^{i_0} \cdots \pi_n(\bar t)^{i_n} T).
\]
This means (\ref{E: cohom formula}) equals
\[
\prod_{r \geq 1} \prod_{\substack{\bar t \in \overline{\bb F}_q^*, r = \deg(\bar t) \\ \hat t = \Teich(\bar t)}} \prod_{i_0 + \cdots + i_n = k} (1 - \pi_0(\bar t)^{i_0} \cdots \pi_n(\bar t)^{i_n} T^r)^{-1/r} = L(\Sym^k \Kl_n / \bb F_q, T),
\]
proving the result.
\end{proof}

%%%-----------------------------------------------------------------------
\section{Cohomology}

%%%-----------------------------------------------------------------------
\subsection{Relative cohomology}\label{S: Rel cohom}

In this section we study the complex $\Omega^\bullet(\c C_{p^k}, D_{t^{p^k}})$ by lifting results on the reduction modulo the uniformizer $\tilde \gamma := \gamma^{1/q}$ of $\f O_q$. We begin by considering the reduction of $\c C_{p^k}$ modulo $\tilde \gamma$ with respect to the orthonormal basis $\{ \gamma^{w(u)} t^{p^k m(u)} x^u \}$; this reduction is defined precisely below.

Set $\oO_{p^k} := \bb F_q[t]$; although there is no $p^k$ on the righthand side, the usefulness of the subscript $p^k$ will be apparent in a moment when we define the reduction map on $\c O_{p^k}$. Next, set $\c R_{p^k} := \bb F_q[ \{ t^r \cdot t^{p^k m(u)} x^u \mid r \geq 0, u \in \bb Z^n\}]$ and define an increasing filtration ($i \geq 0$)
\[
\Fil^i \c R_{p^k} := \{ \sum a_u t^{p^k m(u)} x^u \mid a_u \in \oO_{p^k} \text{ and } w(u) \leq i \}.
\]
Denote its associated graded ring by $\oC_{p^k} := \bigoplus_{i \geq 0} \Fil^i \c R_{p^k} / \Fil^{i-1} \c R_{p^k}$, and observe by Proposition \ref{P: w and m} that multiplication in $\oC_{p^k}$ satisfies
\[
(t^r \cdot t^{{p^k} m(u)} x^u) \cdot (t^s \cdot t^{{p^k} m(v)} x^v) = 
\begin{cases}
t^{r+s} \cdot t^{{p^k} m(u+v)} x^{u+v} & \text{if $u$ and $v$ are cofacial} \\
0  & \text{otherwise.}
\end{cases}
\]
Consequently, the reduction maps 
\[
\c O_{p^k} \rightarrow \oO_{p^k} \quad \text{defined by} \quad a(r) \gamma^{r (n+1)/p^k} t^r \mapsto \overline{a(r)} t^r
\]
with $\overline{a(r)} \in \bb F_q$ the image of $a(r) \in \f O_q$ in the residue field, and
\[
\c C_{p^k} \rightarrow \oC_{p^k} \quad \text{defined by} \quad \eta \gamma^{w(u)} t^{p^k m(u)} x^u \mapsto \overline{\eta} t^{p^k m(u)} x^u
\]
with $\eta \in \c O_{p^k}$ and $\overline{\eta} \in \oO_{p^k}$, are both ring homomorphisms, and further the latter gives an isomorphism $\c C_{p^k} / \tilde \gamma \c C_{p^k} \cong \oC_{p^k}$.

On $\oO_{p^k}$, define the increasing filtration
\[
\Fil^d \oO_{p^k} := 
\begin{Bmatrix}
\text{$\bb F_q$-space generated by $t^r$ } \\ \text{such that $r(n+1)/p^k \leq d$}
\end{Bmatrix}
\]
and on $\oC_{p^k}$, define 
\[
F^{(d,i)} \oC_{p^k} := 
\begin{Bmatrix}
\text{$\bb F_q$-space generated by $t^r \cdot t^{p^k m(u)} x^u$} \\ \text{such that $r(n+1)/p^k \leq d$ and $w(u) \leq i$}
\end{Bmatrix}
\]
Set $\oD_{t^{p^k}, l}^{(1)} := x_l \frac{\partial}{\partial x_l} + K_l(t^{p^k}, x)$, where  $K_l(t, x) := x_l \frac{\partial}{\partial x_l} K(t, x)$. We define the complex $\Omega^\bullet(\oC_{p^k}, \oD_{t^{p^k}}^{(1)})$ in a similar way to $\Omega^\bullet(\c C_{p^k}, D_{t^{p^k}})$, and denote its cohomology by $H^\bullet(\Omega^\bullet(\oC_{p^k}, \oD_{t^{p^k}}^{(1)}))$. By (\ref{E: relate D1 and D}) below, $D_{t^{p^k}} = \oD_{t^{p^k}}^{(1)}$ mod $\tilde \gamma$, and thus the reduction of $\Omega^\bullet(\c C_{p^k}, D_{t^{p^k}})$ mod $\tilde \gamma$ is isomorphic to $\Omega^\bullet(\oC_{p^k}, \oD_{t^{p^k}}^{(1)})$ as $\oO_{p^k}$-modules.

Similar to above, on $\c O_{p^k}$ define the increasing filtration
\[
\Fil^d \c O_{p^k} := 
\begin{Bmatrix}
\text{$\f O_q$-module generated by $\gamma^{r(n+1)/p^k} t^r$} \\ \text{such that $r(n+1)/p^k \leq d$}
\end{Bmatrix},
\]
and on $\c C_{p^k}$ define 
\[
F^{(d,i)} \c C_{p^k} := 
\begin{Bmatrix}
\text{ $\f O_q$-module generated by $\gamma^{(n+1)r/p^k + w(u)} t^r \cdot t^{p^k m(u)} x^u$} \\ \text{such that $r(n+1)/p^k \leq d$ and  $w(u) \leq i$}
\end{Bmatrix}.
\]
We note that the convexity for the functions $m$ and $w$ in the last property of Proposition 3.2 implies that
\[
F^{(d,i)}(\c C_{p^k}) \cdot F^{(e,j)}(\c C_{p^k}) \subseteq F^{(d+e,i+j)}(\c C_{p^k}),
\]
and the analogous relation holds for $F^{(d, i)}(\oC_{p^k})$ as well.

Define $D_{t^{p^k}, l}^{(1)} := x_l \frac{\partial}{\partial x_l} + \gamma K_l(t^{p^k}, x)$. 

\begin{lemma}
We have:
\begin{enumerate}
\item The complex $\Omega^\bullet(\c C_{p^k}, \c D_{t^{p^k}})$ is acyclic except for $H^n$ which is a free $\c O_{p^k}$-module of rank $n+1$. Setting $e_0 := 1$ and $e_i := \gamma^i x_1 x_2 \cdots x_i$ for $1 \leq i \leq n$, then $\{ e_0, e_1, \ldots, e_n \}$ is a basis of this space.
\item For $\xi \in F^{(d, i)}(\c C_{p^k})$ and $0 \leq l \leq \text{min} \{i,n\}$ there exist $a_l \in F^{d+i-l}(\c O_{p^k})$ and $\zeta_l \in F^{(d, i-1)}(\c C_{p^k})$ such that
\begin{equation}\label{E: D1 decomp}
\xi = \sum_{l=0}^{\min \{i,n\}} a_l e_l + \sum_{m=1}^n D_{t^{p^k}, l}^{(1)}(\zeta_m).
\end{equation}
\end{enumerate}
\end{lemma}

\begin{proof}
It was shown in \cite[Theorem 2.5]{MR3572279} that the complex $\Omega^\bullet(\oC_{p^k}, \oD_{t^{p^k}}^{(1)})$ is acyclic except for $H^n$ which is a free $\oO_{p^k}$-module of rank $n+1$, with basis $\bar e_0 := 1$ and $\bar e_i := x_1 x_2 \cdots x_i$ for $1 \leq i \leq n$. It also showed, for $\bar \xi \in F^{(d, i)}(\oC_{p^k})$, there exist $\bar a_l \in \Fil^{i + d - l}(\oO_{p^k})$ and $\bar \zeta_l \in F^{(d, i-1)}(\oC_{p^k})$ such that 
\begin{equation}\label{E: filter on Fq}
\bar \xi = \sum_{l = 0}^{\min \{i,n\}} \bar a_l \bar e_l + \sum_{m=1}^n \oD_{t^{p^k}, m}^{(1)}(\bar \zeta_m).
\end{equation}
A standard lifting argument gives (\ref{E: D1 decomp}).
\end{proof}

\begin{lemma}
We have:
\begin{enumerate}
\item For each $j \geq 1$ and $1 \leq l \leq n$, there exist $\eta_{p^j,l} \in F^{(0, p^j)}(\c C_{p^k})$ such that
\begin{equation}\label{E: relate D1 and D}
D_{t^{p^k}, l}^{(1)} = D_{t^{p^k}, l} + \sum_{j \geq 1} p^{p^j - 1} \eta_{p^j, l}.
\end{equation}
\item For $\xi \in F^{(d, i)}(\c C_{p^k})$ there exist $a_l \in \Fil^{d+i-l}(\c O_{p^k})$, $\zeta_m \in F^{(d, i-1)}(\c C_{p^k})$, and $w_j \in F^{(d, p^j + i -1)}(\c C_{p^k})$ such that
\begin{equation}\label{E: D relation}
\xi = \sum_{l=0}^{\text{min} \{i,n\}} a_l e_l + \sum_{m=1}^n D_{t^{p^k}, l}(\zeta_m) + \sum_{j = 1}^\infty p^{p^j-1} w_j.
\end{equation}
\end{enumerate}
\end{lemma}

\begin{proof}
From the definitions, $D_{t^{p^k}, m}^{(1)} = D_{t^{p^k}, m} - \sum_{j \geq 1} \gamma_j p^j K_m((t^{p^k})^{p^j}, x^{p^j})$. The first assertion comes from rewriting $\gamma_j p^j K_m((t^{p^k})^{p^j}, x^{p^j})$ as follows. First, note that we may write $\gamma_j p^j =  p^{p^j - 1} \tau_j \gamma^{p^j}$ with $ord_p \tau_j = 0$. Then $\eta_{p^j, l} := - \tau_{j} \gamma^{p^j} K_m((t^{p^k})^{p^j}, x^{p^j}) \in F^{(0, {p^j})}(\c C_{p^k})$ gives  (\ref{E: relate D1 and D}).

The second assertion follows quickly from the first: using (\ref{E: D1 decomp}) and (\ref{E: relate D1 and D}),
\begin{align*}
\xi &= \sum_{l=0}^{\text{min} \{i,n\}} a_l e_l + \sum_{m=1}^n D_{t^{p^k}, l}^{(1)}(\zeta_m) \\ 
&= \sum_{l=0}^{\text{min} \{i,n\}} a_l e_l + \sum_{m=1}^n D_{t^{p^k}, m}(\zeta_m) + \sum_{j \geq 1} \sum_{m=1}^n p^{p^j-1} \eta_{p^j, m} \zeta_m.
\end{align*}
Since $ \eta_{p^j,m} \cdot \zeta_m \in F^{(0,p^j)}( \c C_{p^k}) \cdot F^{(d, i-1)}( \c C_{p^k}) \subseteq F^{(d, p^j + i -1)}( \c C_{p^k})$, the result follows by setting $w_j := \sum_{m=1}^n \eta_{p^j,l} \zeta_m$.
\end{proof}

For $b, c \in \bb R$, define the space
\[
\c O_{p^k}(b; c) := \left\{ \sum_{r = 0}^\infty a(r) \gamma^{r(n+1)/p^k} t^r \mid a(r) \in \f O_q \text{ and } \ord_p a(r)\ \geq br + c \right\}.
\]
Denote by $\c F$ the subset of $\prod_{i=1}^\infty \bb Z_{\geq 0}$ where every element $\bj \in \c F$ is of the form $\bj = (j_1, \ldots, j_s, 0, 0, \ldots)$ for some $s \geq 1$, and each $j_i \geq 1$. Define $s(\b j) := \max \{ i \mid j_i > 0 \}.$  When there is no confusion we will write $s$ for $s(\b j)$. Define $\rho(\bj) := p^{j_1} + \cdots + p^{j_s} - s$.

\begin{theorem}\label{T: main rel cohom reduction}
For $\xi \in F^{(d, i)}(\c C_{p^k})$, there exists $C(l, \xi) \in \c O_{p^k}(\frac{n+1}{p^k}; l - i - d)$ such that in $H^n( \Omega^\bullet(\c C_{p^k}, D_{t^{p^k}}))$ we have
\[
\xi = \sum_{l = 0}^n C(l, \xi) e_l   \qquad \mod D_{t^{p^k}}.
\]
The sum on the right is given explicitly by
\[
\sum_{l = 0}^n C(l, \xi) e_l = \sum_{l = 0}^{\min \{i,n\}} a(l, \xi) e_l +  \sum_{\bj \in \c F} \sum_l  p^{\rho(\b j)} a(l, \xi; \b j)e_l 
\]
where $a(l, \xi) \in \Fil^{i + d - l}( \c O_{p^k})$, $a(l, \xi; \b j) \in \Fil^{\rho(\b j) + i + d - l}( \c O_{p^k})$ and the inner sum in the second term on the right runs over $l \leq \min \{n, \rho(\bold{j}) +i\}$ for each $\bj \in \c F$.
\end{theorem}

\begin{proof}
Let $\xi \in F^{(d, i)}(\c C_{p^k})$, then by (\ref{E: D relation}) and replacing the running index $j$ with $j_1$, there exist $a(l) \in \Fil^{d+i-l}(\c O_{p^k})$, $\zeta(m) \in F^{(d, i-1)}(\c C_{p^k})$, and $w_{j_1} \in F^{(d, p^{j_1} + i -1)}(\c C_{p^k})$ such that
\begin{equation}\label{E: Sperber ***}
\xi = \sum_{l=0}^{\min\{i,n\}} a(l) e_l + \sum_{m=1}^n D_{t^{p^k}, m} \zeta(m) + \sum_{j_1 = 1}^\infty p^{p^{j_1}-1} w_{j_1}.
\end{equation}
Since $w_{j_1} \in F^{(d, p^{j_1} + i - 1)}(\c C_{p^k})$, there exist $a(l; j_1) \in F^{d + p^{j_1} + i - 1 - l}(\c O_{p^k})$, $\zeta(m; j_1) \in F^{(d, p^{j_1} + i - 2)}(\c C_{p^k})$, and $w_{j_1, j_2} \in F^{(d, p^{j_1} + p^{j_2} + i - 2)}(\c C_{p^k})$ such that
\[
w_{j_1} = \sum_{l} a(l;j_1) e_l + \sum_{m=1}^n D_{t^{p^k}, m} \zeta(m; j_1) + \sum_{j_2 \geq 1} p^{p^{j_2}-1} w_{j_1, j_2}.
\]
Thus
\[
\sum_{j_1 = 1}^\infty p^{p^{j_1}-1} w_{j_1} = \sum_{j_1 \geq 1}  \sum_l p^{p^{j_1} -1} a(l; j_1) e_l + \sum_{m=1}^n D_{t^{p^k}, m} \left( \sum_{j_1 \geq 1} p^{p^{j_1}-1} \zeta(m;j_1) \right) + \sum_{j_1, j_2 \geq 1} p^{p^{j_1} + p^{j_2} - 2} w_{j_1, j_2}.
\]
with the inner sum in the first term on the right running over $0 \leq l \leq \text{min} \{n, p^{j_1} -1 + i \}$. Repeating this procedure with $w_{j_1, j_2}$ and continuing by induction, we obtain: there exist $a(l, \xi; j_1, \ldots, j_s) \in \Fil^{p^{j_1} + \cdots + p^{j_s} + i + d - l - s}( \c O_{p^k})$ such that
\[
\sum_{j_1 = 1}^\infty p^{p^{j_1}-1} w_{j_1} = \sum_{s = 1}^\infty \sum_{j_1,  \ldots, j_s \geq 1} \sum_{l \leq p^{j_1} + \cdots + p^{j_s} + i - s} p^{p^{j_1} + \cdots + p^{p^{j_s}} - s} a(l, \xi; j_1, \ldots, j_s) e_l  \qquad \text{mod } D_{t^{p^k}}.
\]
In other words: for each $\bold{j} \in \c F$ there exists $a(l, \xi; \bold{j}) \in \Fil^{\rho(\bold{j})+ i + d -l } ( \c O_{p^k})$ such that
\[
\sum_{j_1 = 1}^\infty p^{p^{j_1}-1} w_{j_1} = \sum_{\bold{j} \in \c F} \sum_{l}  p^{\rho(\bold{j})}a(l, \xi; \bold{j}) e_l  \qquad \text{mod } D_{t^{p^k}}.
\]
Since $a(l, \xi; \b j) \in \Fil^{\rho(\b j) + i + d - l }( \c O_{p^k})$, we may write 
\[
p^{\rho(\b j)} a(l, \xi; \b j) =\sum p^{\rho(\b j)} A_k(l, \xi; r) \gamma^{r(n+1)/p^k} t^r,
\]
where $A_k(l, \xi; r) \in \f O_q$ and the sum on the right side runs over $\{r \mid {r(n+1)/p^k \leq \rho(\b j) + i + d - l}\}$. We have then
\[
\ord_p  p^{\rho(\b j)}A_k(l, \xi; r) \geq \rho{(\b j)} \geq r \left( \frac{n+1}{p^k} \right) - d + l - i,
\]
where the rightmost inequality is a consequence of the restrictions in the summation above. There remains the summand $\sum_{l = 0}^{\min\{i,n\}} a(l) e_l$ in (\ref{E: Sperber ***}) with $a(l) \in \Fil^{d+i-l}(\c O_{p^k})$. Since $a(l) \in \Fil^{d+i-l}(\c O_{p^k})$, we may write $a(l) = \sum A(l; r) \gamma^{r(n+1)/p^k} t^r$ where $A(l; r) \in \f O_q$ and the sum runs over $r$ satisfying $r(n+1)/p^k \leq d + i - l$. Thus,
\[
\ord_p A(l; r) \geq 0 \geq \frac{r(n+1)}{p^k} - d - i + l,
\]
where the second inequality comes from the restriction of $r$ in the summation. This proves $C(l, \xi) \in \c O_{p^k}(\frac{n+1}{p^k}; l - i - d)$ as desired.
\end{proof}

%%%-----------------------------------------------------------------------
\subsection{Connection maps on relative cohomology and reduced relative cohomology}\label{SS: connections}

Recall that $K(t, x) := x_1 + \cdots + x_n + t/(x_1 \cdots x_n)$ and $H(t, x) := \sum_{i = 0}^\infty \gamma_i K(t^{p^i}, x^{p^i})$. Define $\partial := t \frac{d}{dt} +  t \frac{d}{dt} H(t, x)$. Since $\partial$ commutes with each $D_{t, l}$ for $1 \leq l \leq n$, it induces a connection map $\partial$ on $H^n( \Omega^\bullet(\c C, D_{t}))$.  Define $\partial^{(1)} :=  t \frac{d}{dt} + \gamma t \frac{\partial}{\partial t} K(t, x)$. Define $G(t, x) := \gamma t \frac{\partial}{\partial t} K(t, x) = \gamma t/(x_1 \cdots x_n)$, and $\partial^{(1)} :=  t \frac{d}{dt} + G(t, x)$. 

\begin{lemma}\label{L: conn}
We have
\begin{enumerate}
\item There exist $\nu_{p^r} \in F^{(0, p^r)}(\c C)$ such that
\begin{equation}\label{E: relate partial1 and partial}
\partial^{(1)} = \partial + \sum_{r \geq 1} p^{p^r - 1} \nu_{p^r}.
\end{equation}
\item For each $0 \leq l \leq n$, we may write in $H^n(\c C, D_t)$
\[
\partial^{(1)}(e_l) = \partial(e_l) + \sum_{\b j \in \c F} \sum_{m \leq \min\{ \rho(\b j) + l +1, n\}} p^{\rho(\b j)} b(m, l; \b j) e_m \qquad \mod D_t
\]
for some $b(m, l; \b j) \in \Fil^{\rho(\b j) + l - m + 1}(\c O)$.
\end{enumerate}
\end{lemma}

\begin{proof}
Observe that $\partial^{(1)} = \partial - \sum_{r = 1}^\infty \gamma_r p^r K_t(t^{p^r}, x^{p^r})$, where $K_t(t, x) := t \frac{\partial}{\partial t} K(t, x) = t / (x_1 \cdots x_n)$. We now rewrite $\gamma_r p^r K_t(t^{p^r}, x^{p^r})$ as follows. First, note that we may write $\gamma_r p^r =  p^{p^r - 1} \tau_r \gamma^{p^r}$ with $\ord_p \tau_r = 0$. Then $\nu_{p^r} := - \tau_{r} \gamma^{p^r} K_t(t^{p^r}, x^{p^r}) \in \Fil^{(0, {p^r})}(\c C)$ gives  (\ref{E: relate partial1 and partial}).

Next, from the first part, we have $\partial^{(1)}(e_l) = \partial(e_l) + \sum_{r \geq 1} p^{p^r - 1} \nu_{p^r} e_l$, with $\nu_{p^r} e_l \in  \Fil^{(0, p^r + l)}(\c C)$. From Theorem \ref{T: main rel cohom reduction}, we may write in $H^n(\c C, D_t)$,
\[
\nu_{p^r} e_l = \sum a(m, r, l) e_m +  \sum_{\b j \in \c F} \sum p^{\rho(\b j)} a(m, r, l; \b j) e_m,
\]
where the first sum runs over $0 \leq m \leq \min \{n, p^r + l\}$ and $a(m, r, l) \in \Fil^{p^r + l - m}(\c O)$, and the second sum runs over $0 \leq m \leq \min \{n, \rho(\b j) + p^r +l \}$ and $a(m, r, l; \b j) \in \Fil^{\rho(\b j) + p^r + l - m}(\c O)$, and $a(m, r, l)$ and $a(m, r, l; \b j)$ equal zero if $m > n$. Thus
\begin{equation}\label{E: some crazy sums}
\sum_{r \geq 1} p^{p^r - 1} \nu_{p^r} e_l = \sum_{r \geq 1} \sum p^{p^r - 1} a(m, r, l) e_m +  \sum_{r \geq 1} \sum_{\bold{j} \in \c F} \sum p^{\rho(\b j) + p^r - 1} a(m, r, l; \b j) e_m
\end{equation}
where the inner sum in the first term on the right runs over $0 \leq m \leq  \min \{n, p^r + l\}$, and the innermost sum in the second term on the right runs over $0 \leq m \leq \min \{n, \rho(\b j) + p^r +l\}$.

Consider now the term $p^{p^r - 1} a(m, r, l)$ in (\ref{E: some crazy sums}), where $a(m, r, l) \in \Fil^{p^r + l - m}(\c O)$. Setting  $\b r := (r, 0, 0, \ldots) \in \c F$, then we may write
\[
p^{p^r - 1} a(m, r, l) = p^{\rho(\b r)} b(m, l; \b r)
\]
with $b(m, l; \b r) := a(m, r, l) \in \Fil^{\rho(\b r) + l - m + 1}(\c O)$. Next, consider the term $p^{\rho(\b j) + p^r - 1} a(m, r, l; \b j)$ in (\ref{E: some crazy sums}) where $a(m, r, l; \b j) \in \Fil^{\rho(\b j) + p^r + l - m}(\c O)$. Writing $\b j = (j_1, \ldots, j_s, 0, 0, \ldots)$, set $\b j' := (j_1, \ldots, j_s, r, 0, 0, \ldots) \in \c F$. Then
\[
p^{\rho(\b j) + p^r - 1} a(m, r, l; \b j) = p^{\rho(\b j')} b(m, l; \b j')
\]
where $b(m, l; \b j') := a(m, r, l; \b j) \in \Fil^{\rho(\b j') + l - m + 1}(\c O)$. That is, the righthand side of (\ref{E: some crazy sums}) may be written in the form
\[
\sum_{\b j \in \c F} \sum p^{\rho(\b j)} b(m, l; \b j) e_m
\]
where $b(m, l; \b j) \in \Fil^{\rho(\b j) + l - m + 1}(\c O)$ and $b(m, l; \b j) = 0$ if $m > n$, and the inner sum runs over $m \leq \min\{n, \rho(\b j) + l + 1\}$
\end{proof}

We now turn to the connection on reduced relative cohomology. Define $\oG(t, x) := t \frac{\partial}{\partial t} K(t, x) = t/(x_1 \cdots x_n)$, and $\bar \partial^{(1)} :=  t \frac{d}{dt} + \oG(t, x)$. Since $\bar \partial^{(1)}$ commutes with each $\oD_{t, l}^{(1)}$ for $1 \leq l \leq n$, it induces a connection map $\bar \partial^{(1)}$ on $H^n( \Omega^\bullet(\oC, \oD_{t}^{(1)}))$. Observe that in $H^n( \Omega^\bullet(\oC, \oD_{t}^{(1)}))$, $\oG$ acts on the basis by $\oG(\bar e_i) = \bar e_{i+1}$ if $0 \leq i < n$ and $\oG(\bar e_n) = t \bar e_0$. Thus, acting on column vectors, the matrix of $\oG$ with respect to the basis $\c B =\{ \bar e_i \}_{i=0}^n$ of $H^n( \Omega^\bullet(\oC, \oD_{t}^{(1)}))$ takes the form
\[
\oG = 
\left(
\begin{array}{ccccc}
 0 &   & \cdots & & t  \\
 1 &  0 &  \cdots & & \\
  &  1 &    &  &\\
  & & \ddots &  &\\
  & & & 1 & 0  
\end{array}
\right)
\]
Note that $\det(\oG - \lambda I) = (-\lambda)^{n+1} + (-1)^{n} t$, and thus the eigenvalues of $\oG$ are $\lambda = t^{1/(n+1)} \zeta_{n+1}^i$ for $i = 0, 1, \ldots, n$, where $\zeta_{n+1}$ is a primitive $(n+1)$-th root of unity in $\overline{\bb F}_p$.

%%%-----------------------------------------------------------------------
\subsection{Reduced symmetric power cohomology}\label{SS: reduced sym cohom}

In this section we study the reduction of the complex $\Omega^\bullet(\c C, D_{t})$ modulo $\tilde \gamma$. Define the reduced symmetric power complex as follows. Set $\oSkH := \Sym^k_{\oO} H^n(\oC, \oD_t^{(1)})$, a free $\oO$-module with basis $\Sym^k \c B := \{ \bar e_0^{i_0} \cdots \bar e_n^{i_n} \mid i_j \in \bb Z_{\geq 0}, i_0 + \cdots + i_n = k \}$. We extend the connection map $\bar \partial^{(1)}$ from Section \ref{SS: connections} to a connection map $\bar \partial^{(1)}_k: \oSkH \rightarrow \oSkH$ via
\[
\bar \partial^{(1)}_k(\bar \xi \bar e_0^{i_0} \cdots \bar e_n^{i_n} ) := t \frac{d}{dt}(\bar \xi) \bar e_0^{i_0} \cdots \bar e_n^{i_n}  + \sum_{l = 0}^n i_j \bar \xi  \cdot \bar e_0^{i_0} \cdots \bar e_j^{i_j - 1} \cdots \bar e_n^{i_n} \cdot \bar \partial^{(1)}(\bar e_j).
\]
Similarly, define the map $\nabla_{\oG}:  \oSkH \rightarrow \oSkH$ by
\[
\nabla_{\oG}(\bar \xi \bar e_0^{i_0} \cdots \bar e_n^{i_n} ) := \sum_{l = 0}^n i_j \bar \xi  \cdot \bar e_0^{i_0} \cdots \bar e_j^{i_j - 1} \cdots \bar e_n^{i_n} \cdot \oG(\bar e_j).
\]
Observe that $\bar \partial^{(1)} = t \frac{d}{dt} + \nabla_{\oG}$. \

Define the complex $\Omega^\bullet(\oSkH, \bar \partial^{(1)}_k)$ by
\[
\Omega^0(\oSkH, \bar \partial^{(1)}_k) := \oSkH  \qquad \text{and} \qquad
\Omega^1(\oSkH, \bar \partial^{(1)}_k) := \oSkH  \frac{dt}{t}
\]
with boundary map $\bar \partial^{(1)}$ defined by $\bar \partial^{(1)}(\eta) := \bar \partial^{(1)}(\eta) \frac{dt}{t}$. Similarly, define the complex $\Omega^\bullet(\oSkH, \nabla_\oG)$. For the rest of this section, we will focus on the latter complex.

From Section \ref{SS: connections}, over the field $\bb K := \overline{\bb F}_q(t^{1/(n+1)})$  we may diagonalize the matrix $\oG$ by $\oG = P \tG P^{-1}$ with
\[
\tG = t^{1/(n+1)}
\left(
\begin{array}{cccc}
1 & & & \\
 & \zeta_{n+1}  &  &\\ 
 & & \zeta_{n+1}^2 & \\
 & &  & \ddots
\end{array}
\right).
\]
Define $\nabla_{\tG}$ similarly to that of $\nabla_{\oG}$, and define $\Sym^k(P)$ on $\oSkH \otimes \bb K$ by
\[
\Sym^k(P)( \bar \xi \bar e_0^{i_0} \cdots \bar e_n^{i_n} ) := \bar \xi \cdot (P \bar e_0)^{i_0} \cdots (P \bar e_n)^{i_n},
\]
and from this define $\Sym^k(P)$ on $\Omega^\bullet(\oSkH\otimes \bb K, \nabla_{\oG})$.

\begin{lemma}\label{L: conj G}
\[
\nabla_{\oG} = \Sym^k P \circ \nabla_{\tG} \circ \Sym^k P^{-1}.
\]
\end{lemma}

\begin{proof}
Since $\oG = P \tG P^{-1}$, we have
\begin{align*}
\nabla_{\oG}( \bar \xi \bar e_0^{i_0} \cdots \bar e_n^{i_n} ) &:= \sum_{j=0}^n i_j \bar \xi \bar e_0^{i_0} \cdots \bar e_j^{i_j-1} \cdots e_n^{i_n} (\oG \bar e_j) \\
&= \sum_{j=0}^n i_j \bar \xi(PP^{-1} \bar e_0)^{i_0} \cdots  (P P^{-1} \bar e_j)^{i_j-1}  \cdots (P P^{-1} \bar e_n)^{i_n} (P \tG P^{-1} \bar e_j) \\
&= \Sym^k P \circ \nabla_{\tG} \circ \Sym^k P^{-1}  ( \bar \xi \bar e_0^{i_0} \cdots \bar e_n^{i_n}).
\end{align*}
\end{proof}

Denote by $\oG_k$ the matrix of $\nabla_\oG$ on $\oSkH$ with respect to the basis $\Sym^k \c B$. Our basic observation is that the nonsingularity of the matrix $\bar G_k$, or equivalently the injectivity of $\nabla_\oG$, will enable us to compute cohomology; in particular proving $H^0(\Omega^\bullet(\oSkH, \nabla_{\oG})) = 0$. We may then lift this to characteristic zero and estimate the Frobenius, as we will see in Section \ref{S: Sym Frob}. A useful sufficient condition for this nonsingularity appeared in the work in Fu-Wan \cite{MR2194679}, and it also appears here. Specifically, define
\[
d_k(n, p) := \# \{ (i_0, \ldots, i_n) \in \bb Z_{\geq 0}^{n+1} \mid i_0 + \cdots + i_n = k \text{ and }  \sum_{j=0}^n i_j \zeta_{n+1}^j \equiv 0 \text{ in } \overline{\bb F}_p \}.
\]

Our main results will require $d_k(n, p) = 0$. Whether it is zero or not appears to be a delicate question. For example, observe that when $k = 2$ and $n+1$ is odd, then $d_2(n, p) = 0$ for all primes $p \geq 3$. However, when $n = 5$, $k = 5$, and $p \equiv 5$ mod 6, then the 6-th cyclotomic polynomial $x^2 - x + 1$ over $\bb Q$ remains irreducible over $\bb F_p$. Then, from
\[
(1 - x + x^2)(1 + 2x + 2x^2) = 1 + x + x^2 + 2x^4,
\]
we see that $(i_0, i_1, \ldots, i_5) = (1, 1, 1, 0, 2, 0)$ is among the points counted by $d_5(5, p)$, making it non-zero in this case. We are indebted to Andy O'desky for the following argument.

\begin{theorem}
Suppose $n + 1$ is a power of a prime other than $p$, and $\gcd(n+1, k) = 1$. Then $d_k(n, p) = 0$ for all $p$ sufficiently large.
\end{theorem}

\begin{proof}
We will show that if $d_k(n, p) \not= 0$ for infinitely many primes $p$, then $\gcd(n+1, k)$ must be greater than 1. Let ${\bf i} = (i_0, \ldots, i_n) \in \bb Z_{\geq 0}^{n+1}$, and denote by $f_{\bf i}(x)$ the polynomial $i_0 + i_1 x + i_2 x^2 + \cdots + i_n x^n$. As there are only finitely many such polynomials with nonnegative integer coefficients whose sum is a fixed $k$, if $d_k(n, p)$ is nonzero for infinitely many primes $p$, then at least one of the algebraic integers $f_{\bf i}(\zeta_{n+1})$ must vanish modulo $p$ for infinitely many primes $p$. However, this means this algebraic integer must itself be zero, which implies that the $(n+1)$-st cyclotomic polynomial $\Phi_{n+1}(x)$ divides $f_{\bf i}(x)$ over $\bb Q$. Setting $x = 1$, we see that $\Phi_{n+1}(1)$ divides $k$. As we are assuming $n+1$ is a prime power $\ell^r$, then $\Phi_{n+1}(1) = \ell$ so $\ell \mid \gcd(n+1, k)$. Hence, $\gcd(n+1, k) \geq 2$.

Note that when $n+1$ is not a prime power, $\Phi_{n+1}(1) = 1$, and so this argument fails to show anything. 
\end{proof}

\begin{corollary}\label{C: G inj}
If $d_k(n, p) = 0$ then $\nabla_{\oG}$ is injective on $\oSkH$, (so $\oG_k$ is nonsingular),  and thus $H^0(\Omega^\bullet(\oSkH, \nabla_{\oG})) = 0$.
\end{corollary}

\begin{proof}
Let $\{ v_i \}_{i=0}^n$ be an eigenbasis of $\tG$ in $H^n(\oC, \oD_t^{(1)}) \otimes \bb K$ with eigenvalues $\zeta_{n+1}^l t^{1/(n+1)}$ for $l = 0, 1, \ldots, n$. Then, in $\oSkH \otimes \bb K$, we have
\begin{align*}
\nabla_{\tG}( v_0^{i_0} \cdots v_n^{i_n}) &= \sum_{j=0}^n i_j v_0^{i_0} \cdots v_j^{i_j - 1} (\widetilde G v_j) \cdots v_n^{i_n} \\
&=  \sum_{j=0}^n i_j v_0^{i_0} \cdots v_j^{i_j - 1} (\zeta_{n+1}^j t^{1/(n+1)} v_j) \cdots v_n^{i_n} \\
&= \left( \sum_{j=0}^n i_j \zeta_{n+1}^j \right) t^{1/(n+1)} (v_0^{i_0} \cdots v_n^{i_n}),
\end{align*}
showing $\nabla_{\tG}$ is injective, and thus by Lemma \ref{L: conj G} we have $\nabla_{\oG}$ injective on $\oSkH$.
\end{proof}

We may now calculate the dimension of $H^1(\Omega^\bullet(\oSkH, \nabla_{\oG}))$ over $\bb F_q$. For $\ui = (i_0, \ldots, i_n) \in \bb Z_{\geq 0}^{n+1}$ and $|\ui| := i_0 + \cdots + i_n$, define $w(\ui) := 0 \cdot i_0 + 1 \cdot i_1 + \cdots + n \cdot i_n$. Set $\bar e^{\ui} := \bar e_0^{i_0} \cdots \bar e_n^{i_n}$. For monomials $t^r \bar e^{\ui} \in \oSkH$, define the weight function $W(r, \ui) := (n+1)r + w(\ui)$. Define a filtration on $\oSkH$ by
\[
\Fil^{N} \oSkH := 
\begin{Bmatrix}
\text{$\bb F_q$-vector space generated by } \\ \text{$t^r e^\ui$  with $W(r, \ui) \leq N$ } 
\end{Bmatrix},
\]
and observe that $\nabla_{\oG}: \Fil^{N} \oSkH \rightarrow \Fil^{N+1} \oSkH \frac{dt}{t}$ is a weight 1 map. We will also use the notations $\oO$ and $\oSkH$ for the associated graded ring and module, and again note that $\nabla_{\oG}$ is a homogeneous map of weight 1. Define
\[
{\overline {\c S}}_k^{(N)}  := 
\begin{Bmatrix}
\text{ $\bb F_q$-vector space generated by } \\ \text{$t^r e^\ui$  with $W(r, \ui) = N$ }
\end{Bmatrix}, 
\]
the homogeneous component of $\oSkH$ of weight $N$. Assuming $d_k(n, p) = 0$, by Corollary \ref{C: G inj} we have an exact sequence
\[
\begin{CD}
0 @>>> \oSkH @>\nabla_{\oG} >> \oSkH \frac{dt}{t} @>>> \oSkH  \frac{dt}{t} / \nabla_{\oG} \oSkH @>>> 0.
\end{CD}
\]
Let $P(T)$ denote the Poincar\'e series of $\oSkH$, and $Q(T)$ the Poincar\'e series of the quotient $\oSkH  \frac{dt}{t} / \nabla_{\oG} \oSkH$. Since $\nabla_{\oG}$ is injective and of weight 1, the Poincar\'e series of $\text{Image}(\nabla_{\oG})$ equals $T P(T)$. The exact sequence shows $Q(T) = (1-T) P(T)$. Now, as an $\bb F_q[t]$-module, 
\[
\oSkH  \cong \bigoplus_{i_0 + \cdots + i_n = k} \bb F_q[t] e_0^{i_0} \cdots e_n^{i_n}.
\]
Since $t$ has weight $n+1$, the Poincar\'e series of $\bb F_q[t]$ is $1 + T^{n+1} + T^{2(n+1)} + \cdots = 1/(1 - T^{n+1})$. The weight of $e_0^{i_0} \cdots e_n^{i_n}$ is $\sum_{j = 0}^n j \cdot i_j$, so define $R(T) := \sum_{i_0 + \cdots + i_n = k} T^{\sum_{j = 0}^n j \cdot i_j}$. Then the Poincar\'e series of $\oSkH$ equals $R(T) / (1 - T^{n+1})$, and
\[
Q(T) = (1 - T) \left( \frac{R(T)}{1-T^{n+1}} \right) = \frac{R(T)}{1 + T + \cdots + T^n}.
\]

\begin{lemma}\label{L: poly}
If $\gcd(n+1, k) = 1$ then $\frac{R(T)}{1 + T + \cdots + T^n}$ is a polynomial.
\end{lemma}

\begin{proof}
Set
\[
I_k := \{ (i_0, \ldots, i_n) \in \bb Z_{\geq 0}^{n+1} \mid i_0 + \cdots + i_n = k \},
\]
and note that
\[
R(T) := \sum_{\ui \in I_k} T^{w(\ui)}.
\]
Setting $d_m := \# \{ \ui \in I_k \mid w(\ui ) = m \}$, we may write
\[
R(T) = \sum_{m=0}^{nk} d_m T^m = \sum_{r=0}^{n} \sum_{j=0}^M d_{(n+1)j+r} T^{(n+1)j+r}
\]
with $M = \lfloor \frac{nk}{n+1} \rfloor$. Consider $R(T)$ modulo $(1+T+T^2 + \cdots + T^n)$. In this case, since $T^{n+1} \equiv 1$,
\begin{equation}\label{E: Q(T)div}
R(T) \equiv \sum_{r=0}^{n} \sum_{j=0}^M d_{(n+1)j+r} T^r \quad \mod(1 + T + \cdots + T^n).
\end{equation}
Setting $C_r := \{ \ui \in I_k \mid w(\ui) \equiv r \text{ mod}(n+1)\}$, then 
\[
\# C_r = \sum_{j=0}^M d_{(n+1)j+r}.
\]
We claim that $\# C_r$ is independent of $r$. Supposing this is true then (\ref{E: Q(T)div}) shows
\[
R(T) \equiv \# C_0 \cdot ( 1 + T + \cdots + T^n ) \equiv 0 \quad \mod(1 + T + \cdots + T^n)
\]
and $\frac{R(T)}{1 + T +T^2 + ...T^n}$ is a polynomial.

Let us now prove the claim. Since $\gcd(k,n+1) =1$ the set 
\[
\{ (k r \text{ mod}(n+1)) \mid r = 0, 1, \ldots, n \}
\]
is a permutation of the set $\{ r \mid r = 0,1,\ldots, n\}$. Thus, we need only show that the cardinality of $C_{kr}$ is independent of $r$ for $r=0, 1, \ldots, n$. To do this, first notice that the map
\[
\ui := (i_0, \ldots, i_n) \mapsto \sigma_1(\ui) := (i_0, i_{n}, i_{n-1}, \ldots, i_1)
\]
is an involution from $C_r$ to $C_{-r}$ since $w(\sigma_1(\ui)) \equiv - w(\ui)$ mod$(n+1)$. Thus, $C_r$ and $C_{-r}$ are bijective. Similarly, we have an involution
\[
\ui := (i_0, \ldots, i_{n}) \mapsto \sigma_2(\ui) := (i_{n}, i_{n-1}, \ldots, i_1, i_0)
\]
from $C_r$ to $C_{-k-r}$ since, using $i_0 + \cdots + i_{n} = k$, we have $w( \sigma_2(\ui)) \equiv -k - w( \ui)$ mod$(n+1)$. Hence, $C_r$ and $C_{-k-r}$ are bijective. Using $\sigma_2$ and $\sigma_1$ successively, we see that
\[
\# C_0 = \# C_{-k} = \# C_{k} = \# C_{-2k} = \# C_{2k} = \cdots
\]
proving the claim.
\end{proof}

\begin{corollary}\label{C: nabla dim}
When $d_k(n, p) = 0$, $H^1(\Omega^\bullet(\oSkH, \nabla_{\oG}))$ is an $\bb F_q$-vector space of dimension $\frac{1}{n+1}\binom{k + n}{n}$.
\end{corollary}

\begin{proof}
Note that
\[
Q(1) = \left. \frac{R(T)}{1+T+ \cdots +T^n} \right|_{T = 1} = \frac{1}{n+1} R(1) = \frac{1}{n+1}\binom{k + n}{n}
\]
proving the dimension formula.
\end{proof}

While the following result is not necessary for our purposes, it may be useful for future endeavors, so we record it here. It states that there exists a basis consisting of ``constant vectors'', that is, a basis of vectors independent of the parameter $t$.

\begin{theorem}\label{T: constant basis}
Suppose $d_k(n, p) = 0$. Then there exists a subset $B_k \subseteq I_k$ (see the proof of Lemma \ref{L: poly} for the definition of $I_k$) such that  $\{ \bar e^\ui := \bar e_0^{i_0} \cdots \bar e_n^{i_n} :  \ui \in B_k \}$ is a basis of $H^1(\Omega^\bullet(\oSkH, \nabla_{\oG}))$.
\end{theorem}

\begin{proof}
We will refer to elements of the form $\bar e^\ui$ as constant vectors. Suppose no such basis of constant vectors exists. Fix a basis $B$, then there exists an element $t^m \bar e^\ui \in B$ which is not equivalent to a sum of constant vectors in $H^1$. Let $t^m \bar e^\ui$ have the largest weight $W(m, \ui)$ among all such non-constant basis terms in $B$; this may not be unique.

As $\nabla_{\oG}$ is a weight-1 map, the vector $t^{m+1} e^\ui$ can be written as $t^{m+1} e^\ui = u + \nabla_{\oG}(\xi)$ with $u$ a sum of constant vectors; this is possible since  if there is any non-constant basis term in $u$, then it would necessarily have a larger weight than $t^m e^\ui$ which we are assuming is impossible. Now, if $t \mid \xi$ then $t \mid \nabla_{\oG}(\xi)$, which means $t \mid u$, implying $u = 0$. If this is the case, then we have $t^{m} e^\ui =  \nabla_{\oG}(\xi/t)$, which is impossible. Thus, $\xi = v + tw$ where $v$  is a linear combination of constant vectors.

Next, write $v = E + E'$ where $E$ is a sum of constant vectors whose terms do not include $\bar e_n$, and $E'$ is a sum of constant vectors whose terms have $e_n$ to some power. Then $Q := \nabla_{\oG}(E)$ is a sum of constant vectors, and $\nabla_{\oG}(E') = P + tR$ where $P$ is a sum of constant vectors, and $R$ is a sum of constant vectors. Thus,
\[
t^{m+1} \bar e^\ui = u + Q + P + tR + t \nabla_{\oG}(w).
\]
Note that we must have $u + Q + P = 0$ since all other terms are divisible by $t$, hence
\[
t^{m+1} \bar e^\ui = tR + t \nabla_{\oG}(w),
\]
which implies $t^m \bar e^\ui = R + \nabla_{\oG}(w)$. That is, $t^m \bar e^\ui$ is equivalent to a sum of constant vectors in cohomology, a contradiction. 
\end{proof}

\begin{corollary}\label{C: nabla_G reduction}
Assume $d_k(n, p) = 0$. For $\bar \xi \in {\overline {\c S}}_k^{(N)}$, there exists $\bar a_{\ui} \in \bb F_q$ with $\ui \in B_k$ and $w(\ui) = N$, and $\bar \zeta \in {\overline {\c S}}_k^{(N-1)}$, such that
\[
\bar \xi = \sum_{\ui \in B_k, w(\ui) = N} \bar a_{\ui} \bar e^{\ui} + \nabla_{\oG}(\bar \zeta).
\]
\end{corollary}

Since $\bar \partial^{(1)} = t \frac{d}{dt} + \nabla_{\overline{G}}$ it follows in a standard way, from Corollaries \ref{C: nabla dim} and \ref{C: nabla_G reduction}, that:

\begin{theorem}
Suppose $d_k(n, p) = 0$. Then the complex $\Omega^\bullet(\oSkH, \bar \partial^{(1)})$ has cohomology $H^0 = 0$, and $H^1$ an $\bb F_q$-vector space of dimension $\frac{1}{n+1}\binom{k+n}{n}$ with basis $\{ \bar e^{\ui} \}_{\ui \in B_k}$. Further, for $\bar \xi \in \Fil^N \oSkH$, there exists $\bar a_{\ui} \in \bb F_q$ with $\ui \in B_k$ with $w(\ui) \leq N$, and $\bar \zeta \in \Fil^{N-1} \oSkH$, such that
\begin{equation}\label{E: reduced sym}
\bar \xi = \sum_{\ui \in \overline{B}_k, w(\ui) \leq N} \bar a_{\ui} \bar e^{\ui} + \bar \partial^{(1)}(\bar \zeta).
\end{equation}
\end{theorem}

%%%-----------------------------------------------------------------------
\subsection{Symmetric power cohomology}

We now lift the results from the previous section. Define $\SkH := Sym^k_{\c O} H^n( \Omega^\bullet(\c C, D_{t})) $, a free $\c O$-module of rank $\binom{n+k}{n}$ with basis $\Sym^k \c B := \{ e^\ui := e_0^{i_0} \cdots e_n^{i_n} : | \ui | = k \}$, where $\ui = (i_0, \ldots, i_n) \in \bb Z_{\geq 0}^{n+1}$ and $|\ui| := i_0 + \cdots + i_n$. Define $w(\ui) := 0 \cdot i_0 + 1 \cdot i_1 + \cdots + n \cdot i_n$. For monomials $\gamma^{(n+1)r} t^r e^{\uu} \in \SkH$, we define the weight function $W(r, \uu) := (n+1)r + w(\uu)$. As every element of $\SkH$ is a finite sum of terms of the form $\xi e_0^{i_0} \cdots e_n^{i_n}$ with $\xi \in \c O$, we may extend the definition of $\partial$ to $\SkH$ via
\[
\partial( \xi e_0^{i_0} \cdots e_n^{i_n} ) := t \frac{d \xi}{dt} e^{\ui} + \sum_{l = 0}^n i_j \xi  \cdot e_0^{i_0} \cdots e_j^{i_j - 1} \cdots e_n^{i_n} \cdot \partial(e_j).
\]
Similarly, define $\partial^{(1)} := t \frac{d}{dt} + \nabla_G$ on $\SkH$.

Define the complex $\Omega^\bullet(\SkH, \partial)$ by
\[
\Omega^0(\SkH, \partial) := \SkH \qquad \text{and} \qquad \Omega^1(\SkH, \partial) := \SkH \frac{dt}{t}
\]
with boundary map $\partial \eta := \partial(\eta) \frac{dt}{t}$. Denote its cohomology by $H^i(\SkH, \partial)$ for $i = 0, 1$. 

Define an increasing filtration
\[
\Fil^{N} \SkH := 
\begin{Bmatrix}
\text{$\f O_q$-module generated by $\gamma^{(n+1)r} t^r e^\ui$ } \\ \text{ with $e^\ui \in \Sym^k \c B$, $r \geq 0$, and $W(r, \ui) \leq N$ } 
\end{Bmatrix}.
\]
Observe that the reduction of the complex $\Omega^\bullet(\SkH, \partial)$ mod $\tilde \gamma$ with respect to the orthonormal basis $\{\gamma^{(n+1)r} t^r e^\ui\}$ is isomorphic to the complex $\Omega^\bullet(\oSkH, \bar \partial^{(1)})$ from Section \ref{SS: reduced sym cohom}. As a consequence, we have:

\begin{lemma}\label{L: lift partial(1)}
For $\xi \in \Fil^N \SkH$, there exists $a_{\ui} \in \f O_q$ with $\ui \in B_k$ with $w(\ui) \leq N$, and $\zeta \in \Fil^{N-1} \SkH$, such that
\[
\xi = \sum_{\ui \in B_k, w(\ui) \leq N}  a_{\ui}  e^{\ui} + \partial^{(1)}(\zeta).
\]
\end{lemma}

\begin{proof}
This follows from a standard lifting argument of (\ref{E: reduced sym}).
\end{proof}

\begin{lemma}\label{L: w_jr}
For $\xi \in \Fil^N \SkH$, there exists $a_{\ui} \in \f O_q$, $\zeta \in \Fil^{N-1} \SkH$, and  $w(\bj) \in \Fil^{N + \rho(\bj)} \c S_k$  for  $\bj \in \c F,$ such that 
\begin{equation}\label{E: start of recursion}
\xi = \sum_{\ui \in B_k, w(\ui) \leq N}  a_{\ui}  e^{\ui} + \partial(\zeta)  +  \sum_{\bj \in \c F} p^{\rho(\bj)} w(\bj).
\end{equation}
\end{lemma}

\begin{proof}
From Lemma \ref{L: lift partial(1)}, there exists $a_{\ui} \in \f O_q$ with $\ui \in B_k$ with $w(\ui) \leq N$, and $\zeta \in \Fil^{N-1} \SkH$, such that
\[
\xi = \sum_{\ui \in B_k, w(\ui) \leq N}  a_{\ui}  e^{\ui} + \partial^{(1)}(\zeta).
\]
Write $\zeta = \sum C(r, \ui) \gamma^{(n+1)r} t^r e^\ui$ where $C(r, \ui) \in \f O_q$ and the sum runs over $r \geq 0$ and $\ui$ such that $(n+1)r + w(\ui) \leq N-1$. In particular, for each summand $W(t^r e^{\ui}) \leq N-1$. Then
\begin{align*}
\partial^{(1)}( \gamma^{(n+1)r} t^r e^\ui) &= t \frac{d}{dt}( \gamma^{(n+1)r} t^r) e^\ui + \gamma^{(n+1)r} t^r \partial^{(1)}(e^\ui) \\
&= t \frac{d}{dt}( \gamma^{(n+1)r} t^r) e^\ui + \gamma^{(n+1)r} t^r \sum_{l=0}^n i_l e_0^{i_0} \cdots e_l^{i_l - 1} \cdots e_n^{i_n} \cdot \partial^{(1)}(e_l).
\end{align*}
From Lemma \ref{L: conn}, we may write
\[
\partial^{(1)}(e_l) = \partial(e_l) + \sum_{\b j \in \c F} \sum_{m \leq \rho(\b j) + l + 1} p^{\rho(\b j)} b(m, l; \b j) e_m \qquad \mod D_t
\]
for some $b(m, l; \b j) \in \Fil^{\rho(\b j) + l - m + 1}(\c O)$ and $b(m, l; \b j) = 0$ if $m > n$. Plugging this into the above, we have 

\[
\partial^{(1)}( \gamma^{(n+1)r} t^r e^\ui) = \partial( \gamma^{(n+1)r} t^r e^\ui) + \sum_{\bj \in \c F} p^{\rho(\bj)} w(\bj)
\]
where
\[
w(\bj) := \sum_{l= 0}^n \sum_{m \leq \rho(\bj) + l + 1} \gamma^{(n+1)r} t^r i_l \cdot b(m, l; \bj) e_0^{i_0} \cdots e_l^{i_l - 1} \cdots e_n^{i_n} e_m
\]
with the weight of a summand satisfying 
\begin{align*}
W(b(m, l; \bj) \cdot t^r e_0^{i_0} \cdots e_l^{i_l - 1} \cdots e_n^{i_n} e_m) &\leq (\rho(\b j) + l - m + 1) + (N-1-l)  + m \\
&= N + \rho(\b j),
\end{align*}
so that $w(\bj) \in \Fil^{\rho(\bj) + N} \c S_k$. The result follows.
\end{proof}

\begin{theorem}\label{T: sym power cohom}
Suppose $d_k(n, p) = 0$.Then $H^0( \SkH, \partial ) = 0$, and $H^1( \SkH, \partial )$ is a free $\f O_q$-module of rank $\frac{1}{n+1} \binom{n+k}{n}$ with basis $B_k$. Further, for $\xi \in \Fil^N \SkH$ there exists $C(\ui, \xi) \in \f O_q$ satisfying $\ord_p C(\ui, \xi) \geq w(\ui) - N$ such that in $H^1( \SkH, \partial )$
\[
\xi = \sum_{\ui \in B_k} C(\ui, \xi) e^\ui  \qquad \text{mod } \partial.
\]
\end{theorem}

\begin{proof}
Since the reduction of the complex $\Omega^\bullet(\SkH, \partial)$ mod $\tilde \gamma$ with respect to the orthonormal basis $\{\gamma^{(n+1)r} t^r e^\ui\}$ is isomorphic to the complex $\Omega^\bullet(\oSkH, \bar \partial^{(1)})$, and $\Omega^\bullet(\oSkH, \bar \partial^{(1)})$ has cohomology $H^0(\oSkH, \bar \partial^{(1)}) = 0$, and $H^1(\oSkH, \bar \partial^{(1)})$ an $\bb F_q$-vector space of dimension $\frac{1}{n+1}\binom{k+n}{n}$ with basis $\{ \bar e^{\ui} \}_{\ui \in B_k}$, a standard lifting argument shows $H^0( \SkH, \partial ) = 0$, and $H^1( \SkH, \partial )$ is a free $\f O_q$-module of rank $\frac{1}{n+1} \binom{n+k}{n}$ with basis $B_k$. 

Let $\xi \in \Fil^N \SkH$. By Lemma \ref{L: w_jr}, there exists $a_{\ui} \in \f O_q$ and $w(\bj_1) \in \Fil^{N + \rho(\bj)} \c S_k$ for each $\bj_1 \in \c F$ such that 
\[
\xi = \sum_{\ui \in B_k, w(\ui) \leq N}  a_{\ui}  e^{\ui}  +  \sum_{\bj_1 \in \c F} p^{\rho(\bj_1)} w(\bj_1) \qquad \mod \partial.
\]
We now recurse on $\bj_1$. By Lemma \ref{L: w_jr} there exist $a_\ui(\bj_1) \in \f O_q$ and $w(\bj_1, \bj_2) \in  \Fil^{N + \rho(\bj_1) + \rho(\bj_2)} \c S_k$ for each $\b j_2 \in \c F$ such that
\[
w(\bj_1) = \sum_{\ui \in B_k, w(\ui) \leq N + \rho(\bj_1)} a_\ui(\bj_1) e^\ui + \sum_{\bj_2 \in \c F} p^{\rho(\bj_2)} w(\bj_1, \bj_2).
\]
Then
\[
\xi = \sum_{\ui \in B_k, w(\ui) \leq N}  a_{\ui}  e^{\ui}  + \sum_{\bj_1 \in \c F} \sum_{\ui \in B_k, w(\ui) \leq N + \rho(\bj_1)} p^{\rho(\bj_1)} a_\ui(\bj_1) e^\ui + \sum_{\bj_1, \bj_2 \in \c F} p^{\rho(\bj_1) + \rho(\bj_2)} w(\bj_1, \bj_2).
\]
Continuing via induction, we have
\[
\xi = \sum_{\ui \in B_k, w(\ui) \leq N}  a_{\ui}  e^{\ui} + \sum_{s \geq 1} \sum_{\bj_1, \ldots, \bj_s \in \c F} \sum_{\substack{\ui \in B_k \\ w(\ui) \leq N + \rho(\bj_1) + \cdots + \rho(\bj_s)}} p^{\rho(\bj_1) + \cdots + \rho(\bj_s)} a_\ui(\bj_1, \ldots, \bj_s) e^\ui.
\]
The last summation runs over all $\ui$ such that $w(\ui) - N \leq  \rho(\bj_1) + \cdots + \rho(\bj_s)$, thus
\begin{align*}
\ord_p p^{\rho(\bj_1) + \cdots + \rho(\bj_s)} a_\ui(\bj_1, \ldots, \bj_s) &\geq \rho(\bj_1) + \cdots + \rho(\bj_s) \\
&\geq w(\ui) - N
\end{align*}
proving the theorem.
\end{proof}

%%%-----------------------------------------------------------------------
\section{Frobenius}\label{S: Sym Frob}

%%%-----------------------------------------------------------------------
\subsection{Estimates on the relative Frobenius}\label{S: Rel Frob}

Recall from Section \ref{S: cohom formula} that $\alpha_1(t^{p^r}): \c C_{p^r} \rightarrow \c C_{p^{r+1}}$, and since $p D_{l, t^{p^{r+1}}} \circ \alpha_1(t^{p^r}) = \alpha_1(t^{p^r}) \circ D_{l, t^{p^r}}$ for each $1 \leq l \leq n$, we may define a chain map $\Frob^\bullet(\alpha_1(t^{p^r})):  \Omega^\bullet(\c C_{p^r}, D_{t^{p^r}}) \rightarrow \Omega^\bullet(\c C_{p^{r+1}}, D_{t^{p^r+1}})$ by acting  on each factor $\c C_{p^r}\frac{dx_{j_1}}{x_{j_1}} \wedge \cdots \wedge \frac{dx_{j_i}}{x_{j_i}}$ of $\Omega^i(\c C_{p^r}, D_{t^{p^r}})$  via
\begin{equation}\label{E: chain map}
\Frob^i(\alpha_1(t^{p^r}))(\xi \frac{dx_{j_1}}{x_{j_1}} \wedge \cdots \wedge \frac{dx_{j_i}}{x_{j_i}})  = p^{n-i} \alpha_1(t^{p^r})(\xi) \frac{dx_{j_1}}{x_{j_1}} \wedge \cdots \wedge \frac{dx_{j_i}}{x_{j_i}}.
\end{equation}
For these complexes, $H^i = 0$ for all $i \not= n$. Set $\bar \alpha_1(t^{p^r}) := \Frob^n(\alpha_1(t^{p^r})): H^n(\c C_{p^r}, D_{t^{p^r}}) \rightarrow H^n(\c C_{p^{r+1}}, D_{t^{p^{r+1}}})$.
For notational convenience, we will write $\alpha_1$ for $\alpha_1(t)$. For each $0 \leq i \leq n$, with $e_i := \gamma^i x_1 x_2 \cdots x_i$, write
\begin{equation}\label{E: Frob on ei}
\alpha_1(e_i) = \sum_{u \in \bb Z^n} A(u, i) \gamma^{w(u)} t^{p m(u)} x^u \qquad \text{with } A(u, i) \in \c O_p.
\end{equation}

\begin{lemma}\label{L: chain level est}
For every $0 \leq i \leq n$ and $u \in \bb Z^n$, $A(u, i) \in \c O_p(\frac{n+1}{p}; w(u))$.
\end{lemma}

\begin{proof}
Recall from Section \ref{S: cohom formula} that $\alpha_1 := \psi_x \circ F(t, x)$, where $F(t, x) := \theta(x_1) \cdots \theta(x_n) \cdot \theta(t/(x_1 \cdots x_n))$, and $\theta(T) :=  E(\gamma T) =\sum_{j=0}^{\infty} \theta_j T^j$ has coefficients $\theta_j \in \bb Z_p[\gamma]$ satisfying $\text{ord}_p(\theta_j) \geq \frac{j}{p-1}$. Write $F(t, x) = \sum_{u \in \bb Z^n} B(u) x^u$ and note that
\[
B(u) := \sum \theta_{m_1} \theta_{m_2} \cdots \theta_{m_n} \theta_l t^l,
\]
where the sum runs over all $(m_1, \ldots, m_n, l) \in \bb Z_{\geq 0}^{n+1}$ such that $l \geq m(u)$ and $m_j - l = u_j$ for $1 \leq j \leq n$. It follows that $A(u, i) = \gamma^{i-w(u)} B(pu-e_i) t^{-p m(u)}$, where $pu - e_i$ means $pu - (1, 1, \ldots, 1, 0, \ldots, 0)$ and $i$ entries in the latter are 1. Then
\begin{align*}
\gamma^{i-w(u)} B(pu-e_i) t^{-p m(u)} &= \sum \theta_{m_1} \theta_{m_2} \cdots \theta_{m_n} \theta_l \gamma^{i - w(u)} t^{l - p m(u)} \\
&= \sum \left( \theta_{m_1} \theta_{m_2} \cdots \theta_{m_n} \theta_l \gamma^{i - w(u)} \gamma^{-(l-pm(u))(n+1)/p} \right) \gamma^{(l-pm(u))(n+1)/p)} t^{l - pm(u)},
\end{align*}
where both sums runs over all $(m_1, \ldots, m_n, l) \in \bb Z_{\geq 0}^{n+1}$ such that $l \geq m(pu - e_i)$ and $m_k - l = p u_k - 1$ for $1 \leq k \leq i$ and $m_k - l = p u_k$ for $i < k \leq n$.  Using this, and $w(u) = \sum_{i=1}^n u_i + (n+1) m(u)$, we see that
\begin{align}\label{E: ord est}
\ord_p  \left( \theta_{m_1} \theta_{m_2} \cdots \theta_{m_n} \theta_l \gamma^{i - w(u)} \gamma^{-(l-pm(u))/p} \right) &\geq \frac{1}{p-1}\left( \sum_{k=1}^n m_k + l + i - w(u) \right) - \frac{l-pm(u)}{p(p-1)}(n+1) \notag \\
&= \frac{1}{p-1}\left( p \sum_{k=1}^n u_k + (n+1)l  - w(u) \right) - \frac{l-pm(u)}{p(p-1)}(n+1) \notag \\
&= \frac{1}{p-1}\left( p(w(u) - (n+1)m(u)) + (n+1)l - w(u) \right) - \frac{l-pm(u)}{p(p-1)}(n+1)  \notag \\
&= w(u) + \frac{n+1}{p}(l - pm(u)).
\end{align}
Hence, $A(u, i) \in \c O_p(\frac{n+1}{p}; w(u))$.
\end{proof}

By definition, $\bar \alpha_1 := \Frob^n(\alpha_1(t)): H^n(\c C, D_{t}) \rightarrow H^n(\c C_{p}, D_{t^{p}})$ and so there exist $\tilde A(j, i) \in \c O_p$ such that
\[
\bar \alpha_1(e_i) = \sum_{j = 0}^n \tilde A(j, i) e_j \qquad \mod D_{t^p}
\]

\begin{theorem}\label{T: tilde A matrix}
$\tilde A(j,i) \in \c O_p(\frac{n+1}{p}; j)$. 
\end{theorem}

\begin{proof}
Since $\gamma^{w(u)} t^{p m(u)} x^u \in \Fil^{(0, w(u))} \c C_p$, by Theorem \ref{T: main rel cohom reduction} there exist $C(j, u) \in \c O_p(\frac{n+1}{p}; j - w(u))$ such that 
\[
\gamma^{w(u)} t^{p m(u)} x^u = \sum_{j = 0}^n C(j, u) e_j \qquad \mod D_{t^p}.
\] 
It follows from (\ref{E: Frob on ei}) and Lemma \ref{L: chain level est} that $\tilde A(j, i) = \sum_{u \in \bb Z^n} A(u, i) C(j, u) \in  \c O_p(\frac{n+1}{p}; j)$.
\end{proof}

For future reference, we record here the following.

\begin{corollary}
$\bar \alpha_1(e_0) \equiv e_0 \mod (t, \gamma^{n+1})$. 
\end{corollary}

\begin{proof}
By Theorem \ref{T: tilde A matrix}, $\ord_p \tilde A(j, i) \geq 1$ for $j \not= 0$. Thus, we need only focus on $\tilde A(0, 0)$, and in this case, since $C(0, 0) = 1$,
\[
\tilde A(0, 0) = A(0, 0) + \sum_{u \in \bb Z^n, u \not= 0} A(u, i) C(j, u).
\]
As $u \not= 0$, $\ord_p A(u, i) C(j, u) \geq w(u) \geq 1$. The result follows since
\[
A(0, 0) = B(0) = 1 + \sum_{l \geq 1} \theta_l^{n+1} t^l.
\]
\end{proof}

%%%-----------------------------------------------------------------------
\subsection{Estimates on the symmetric power Frobenius}

Let $\c S_{k, p^r} := \Sym^k_{\c O_{p^r}} H^n( \Omega^\bullet(\c C_{p^r}, D_{t^{p^r}}))$,  a free $\c O_{p^r}$-module of rank $\binom{n+k}{n}$ with basis $\Sym^k \c B := \{ e^\ui : |\ui| = k\}$ where $\ui = (i_0, \ldots, i_n) \in \bb Z_{\geq 0}^{n+1}$. To ease notation, define $I_k := \{ \ui = (i_0, \ldots, i_n) \in \bb Z_{\geq 0}^{n+1} \mid |\ui| := i_0 + \cdots + i_n = k \}$. For each $\ui \in I_k$, write
\begin{equation}\label{E: Sym alpha}
\Sym^k \bar \alpha_1(e^\ui) = \sum_{\uj \in I_k} \tilde A^{(k)}(\uj, \ui) e^\uj
\end{equation}
where $\tilde A^{(k)}(\uj, \ui) \in \c O_{p}$. 

\begin{lemma}\label{L: Sym Frob}
For every $\ui, \uj \in I_k$, $\tilde A^{(k)}(\uj, \ui) \in \c O_p(\frac{n+1}{p}; w(\uj))$.
\end{lemma}

\begin{proof}
We proceed by induction on $k$, the case $k = 1$ being Theorem \ref{T: tilde A matrix}. Consider $e_0^{i_0} \cdots e_m^{i_m}$ such that $0 \leq m \leq n$ and $i_m \geq 1$, and set $\ui := (i_0, \ldots, i_m, 0, \ldots, 0)$. Also, set $\ui' := (i_0, \ldots, i_{m-1}, i_m - 1, 0, \ldots, 0)$, then
\[
\Sym^{k-1} \bar \alpha_1(e^{\ui'}) = \sum_{\uj' \in I_{k-1}} \tilde A^{(k-1)}(\uj', \ui') e^{\uj'}
\]
satisfies the induction hypothesis. Now,
\begin{align*}
\Sym^k  \bar \alpha_1(e^\ui) &= \Sym^{k-1} \bar \alpha_1(e^{\ui'}) \otimes \bar \alpha_1(e_m) \\
&= \left(  \sum_{\uj' \in I_{k-1}} \tilde A^{(k-1)}(\uj', \ui') e^{\uj'} \right) \cdot \left( \sum_{l=0}^n \tilde A(l, m) e_l \right) \\
&= \sum_{0 \leq l \leq n \text{ and } \uj' \in I_{k-1}} \tilde A^{(k-1)}(\uj', \ui') \tilde A(l,m) e_0^{j_0'} \cdots e_l^{j_l' + 1} \cdots e_n^{j_n'} \\
&= \sum_{\uj \in I_k} \tilde A^{(k)}(\uj, \ui) e^{\uj}
\end{align*}
where
\[
\tilde A^{(k)}(\uj, \ui) = \sum_{\substack{0 \leq l \leq n \text{ and } \uj' \in I_{k-1} \\ (j_0', \ldots, j_l' + 1, \ldots, j_n') = \uj}} \tilde A^{(k-1)}(\uj', \ui') \tilde A(l,m) 
\]
Under the condition in the last summation, note that $w(\uj') + l = w(\uj)$. Thus, by the induction hypothesis and Theorem \ref{T: tilde A matrix},
\[
\tilde A^{(k-1)}(\uj', \ui') \tilde A(l,m)  \in \c O_p(\tfrac{n+1}{p}; w(\uj')) \cdot \c O_p(\tfrac{n+1}{p}; l) \subseteq \c O_p(\tfrac{n+1}{p}; w(\uj))
\]
proving the assertion. 
\end{proof}

Assume now that $d_k(n, p) = 0$. Set $\beta_{k, 1} := \psi_t \circ \Sym^k(\bar \alpha_1): \c S_{k, t} \rightarrow \c S_{k, t}$, and since $\beta_{k, 1} \partial_t = p \partial_t \beta_{k, 1}$ there is an induced map $\bar \beta_{k, 1}$ on $H^1(\c S_{k, t}, \partial_t)$. For each $\ui, \uj \in B_k$ (where $B_k$ is the constant basis from Theorem \ref{T: constant basis}), there exist $B(\uj, \ui) \in \f O_q$ such that
\[
\bar \beta_{k, 1}(e^\ui) = \sum_{\uj \in B_k} B(\uj, \ui) e^{\uj}.
\]

\begin{theorem}\label{T: Sym Frob Estimates}
Suppose that $d_k(n, p) = 0$. For each $\ui, \uj \in B_k$, we have $\ord_p B(\uj, \ui) \geq w(\uj)$.
\end{theorem}

\begin{proof}
From (\ref{E: Sym alpha}) and Lemma \ref{L: Sym Frob}, we have $\Sym^k \bar \alpha_1(e^\ui) = \sum_{\uu \in I_k} \tilde A^{(k)}(\uu, \ui) e^\uu$ with $\tilde A^{(k)}(\uu, \ui) \in \c O_p(\frac{n+1}{p}; w(\uu))$. Writing $\tilde A^{(k)}(\uu, \ui) = \sum_{l \geq 0} \tilde A^{(k)}(\uu, \ui; l) \gamma^{(n+1)(l/p)} t^l$, then
\[
\beta_{k, 1}( e^\ui ) = \sum_{l \geq 0, \uu \in I_k} \tilde A^{(k)}(\uu, \ui; pl) \gamma^{(n+1)l} t^l e^\uu.
\]
Since $\gamma^{(n+1)l} t^l e^\uu \in \Fil^{W(l, \uu)} \c S_k$, where $W(l, \uu) = (n+1)l + w(\uu)$, by Theorem \ref{T: sym power cohom}, there exist $C(\uj; l, \uu) \in \f O_q$ satisfying $\ord_p C(\uj; l, \uu) \geq w(\uj) - (n+1)l - w(\uu)$ such that $\gamma^{(n+1)l} t^l e^\uu = \sum_{\uj \in B_k} C(\uj; l, \uu) e^{\uj}$ modulo $\partial_t$, and so $\beta_{k, 1}( e^\ui) = \sum_{\uj \in B_k} B(\uj, \ui) e^\uj$ 
with
\begin{equation}\label{E: B sum}
B(\uj, \ui)  = \sum_{l \geq 0, \uu \in I_k} \tilde A^{(k)}(\uu, \ui; pl) C(\uj; l, \uu).
\end{equation}
The result follows from
\[
\ord_p \tilde A^{(k)}(\uu, \ui; pl) C(\uj; l, \uu)  \geq \left(\tfrac{n+1}{p}\right)(pl) + w(\uu) + \left(w(\uj) - (n+1)l - w(\uu) \right) = w(\uj).
\]
\end{proof}

Recall $R(T) := \sum_{i_0 + \cdots + i_n = k} T^{i_1 + 2 i_2 + \cdots + n i_n}$. Write $R(T) / (1 + T + \cdots + T^n) = \sum_{i=0}^\infty h_i T^i$. Define the Hodge polygon as the lower convex hull of the points $(0, 0)$ and $( \sum_{i = 0}^N h_i, \sum_{i = 0}^N i h_i )$ for $N = 0, 1, 2, \ldots$. Note, this is the same as the $q$-adic Newton polygon of $\prod (1 - q^i T)^{h_i}$.

\begin{theorem}\label{T: NP over HP}
Suppose $d_k(n, p) = 0$. The $q$-adic Newton polygon of $L(Sym^k Kl_n / \bb F_q, T)$ lies on or above the Hodge polygon.
\end{theorem}

\begin{proof}
Observe that
\begin{align*}
\beta_{k, 1}^a &= \left( \psi_t \circ Sym^k \bar \alpha_1(t) \right) \circ \cdots \circ \left( \psi_t \circ Sym^k \bar \alpha_1(t) \right) \\
&= \psi_t^a \circ Sym^k \left( \bar \alpha_1(t^{p^{a-1}}) \bar \alpha_1(t^{p^{a-2}}) \cdots \bar \alpha_1(t) \right) \\
&= \psi_t^a \circ Sym^k \bar \alpha_a \\
&= \beta_k,
\end{align*}
and so $\bar \beta_{k, 1}^a = \bar \beta_k$. The result now follows from Theorem \ref{T: Sym Frob Estimates} and Dwork's argument \cite[Section 7]{MR0188215}. 
\end{proof}

%%%-----------------------------------------------------------------------
\bibliographystyle{amsplain}
\bibliography{../../References/References}

\end{document}